\documentclass[12pt,a4paper]{amsart}
\usepackage{amsfonts}
\usepackage{amsthm}
\usepackage{amsmath}
\usepackage{amscd}
\usepackage[latin2]{inputenc}
\usepackage{t1enc}
\usepackage[mathscr]{eucal}
\usepackage{indentfirst}
\usepackage{graphicx}
\usepackage{graphics}
\usepackage{pict2e}
\usepackage{epic}
\usepackage[colorlinks,linkcolor=red,anchorcolor=blue,citecolor=green,CJKbookmarks=True]{hyperref}
\usepackage{verbatim}
\usepackage[margin=2.9cm]{geometry}
\usepackage{theoremref}
\usepackage{epstopdf}
\usepackage{chngcntr}

 \newtheorem{Th}{Theorem}[section]
 \newtheorem{Lemma}[Th]{Lemma}
\newtheorem{Cor}[Th]{Corollary}

 \newtheorem{Rem}[Th]{Remark}
 \newtheorem*{rem}{Remark}

\newtheorem{Def}[Th]{Definition}
\usepackage{soul}

\begin{document}

\bibliographystyle{amsalpha}

\title{$SO(2)$ symmetry of the translating solitons of the mean curvature flow in $\mathbb{R}^4$}

\begin{abstract}
In this paper, we prove that the translating solitons of the mean curvature flow in $\mathbb{R}^4$ which arise as blow-up limit of embedded, mean convex mean curvature flow must have $SO(2)$ symmetry.
\end{abstract}

\author{Jingze Zhu}
\address{Department of Mathematics, Columbia University, New York, NY 10027}
\maketitle

\section{Introduction}\label{sec:intro}

Translating solitons are an important class of solutions to the mean curvature flow. They often arise as blow-up limit of type II singularities, following Hamilton's blow-up procedure \cite[Section 16]{hamilton1993formations}.

Before we go into it, it's worthwhile to discuss 
 a wider class of solutions called ancient solutions, which are defined on $(-\infty, T)$. As we will see later, our study of the translating solitons depend on the understanding of the ancient solutions.  The ancient solutions to the geometric flow  were brought to attention by the famous work in Ricci flow by Perelman \cite{perelman2002entropy}. He  showed
that the high curvature regions of the 3 dimensional Ricci flow are modeled on
$\kappa$ solutions, which are ancient, $\kappa$-noncollapsed, and have bounded positive curvature. $\kappa$ solutions therefore arise as limits of general blow-up process.

Similar results are available in the mean convex mean curvature flow. Namely, before the first singular time the high curvature regions of the embedded, mean convex mean curvature flow are modeled on the convex, non-collapsed ancient solutions.
This was proved by the fundamental work of White \cite{white2000size, white2003nature}  and later streamlined by Haslhofer--Kleiner \cite{haslhofer2017mean}.
There are also various works by Huisken--Sinestrari \cite{huisken1999convexity}, Sheng--Wang \cite{sheng2009singularity} and
Andrews \cite{andrews2012noncollapsing} concerning the non-collapsing property, convexity estimate.

  It is natural to ask what these ancient solutions look like. We want to mention  the recent breakthrough classification results of the convex non-collapsed ancient solutions, assuming uniform 2-convexity additionally.
  In the noncompact case,
  Brendle--Choi \cite{brendle2019uniqueness, brendle2018uniqueness} showed that the strictly convex solution must be Bowl soliton. In the compact case, Angenent--Daskalopoulos--Sesum \cite{angenent2019unique} \cite{angenent2020uniqueness} showed that there is only one ancient oval (up an Eulidean isometry and scaling) that is not a shrinking sphere,
  whose existence was exploited by Haslhofer--Hershkovits  \cite{haslhofer2013ancient}.
  If the solution is not strictly convex, then Hamilton's maximum principle \cite[Section 4]{hamilton1986four} implies that the solution is a product of a line and a uniformly convex ancient solution, therefore must be a shrinking cylinder by \cite{huisken1984flow}.

  However, it's still challenging to understand the convex non-collapsed ancient solutions without the uniform 2-convexity assumption.
  One reason comes from the possibly large number of solutions. There is a family of convex graphical translating solutions known as flying wings (see the paper of Bourni--Langford--Tinaglia \cite{bourni2020existence} and Hoffman--Ilmanen--Martin--White \cite{Hoffman2019}).
  Asymptotically they look like two pieces of slightly dilated grim reaper$\times\mathbb{R}$ tilted by opposite angle (the dilation factor is related to the angle). The flying wings are in the slab region, so they can't be non-collapsed and don't fit into our picture.
  But it is expected that in dimensions higher than 3 we have similar objects which asymptotically look like two symmetric pieces of tilted Bowl$\times\mathbb{R}$ and they should be convex and non-collapsed.
  In fact, Hoffman--Ilmanen--Martin--White \cite{Hoffman2019} also constructed a family
  of entire graph translators.

  To make a step forward, we will show that the translating solitons in $\mathbb{R}^4$ that may arise as blow-up limit (by Hamilton's process) of embedded, mean convex mean curvature flow must have $SO(2)$ symmetry, which means that they are invariant under $SO(2)$ action.
  By the previous work mentioned above, any of such solutions must be convex, non-collapsed. Moreover, Hamilton's blow-up process \cite{hamilton1993formations} ensures that the maximal mean curvature is attained. Here is the main theorem:

  \begin{Th}\label{main theorem}
        Suppose that $M^3\subset\mathbb{R}^4$ is a complete, noncollapsed, convex smooth translating soliton of mean curvature flow with a tip that attains maximal mean curvature. Then $M$ has $SO(2)$ symmetry.
    \end{Th}

  The entire translating solutions constructed by Wang in \cite{wang2011convex} have $SO(2)$ symmetry in $\mathbb{R}^4$ (i.e. entire graphs over $\mathbb{R}^3$). These solutions all have a tip that attains maximal mean curvature and they are believed to be noncollapsed, since they are entire solutions. Therefore  Theorem \ref{main theorem} is in a way consistent with Wang's construction.

  The dimension restriction can be replaced by the condition of uniform 3-convexity. This way we have a natural generalization in higher dimensions, which should be true with minor changes, as \cite{brendle2018uniqueness} demonstrated in the uniform 2-convex setting. In $\mathbb{R}^4$ of course the uniform 3-convexity is a vacuous condition.

  This paper
  will be organized as follows:

  We begin with  preliminaries and notations in section \ref{section preliminary}.
  In section \ref{section symmetry improvment}, we discuss 
  the symmetry improvement, which are inspired by the Neck Improvement Theorem first proved in \cite[Theorem 4.4]{brendle2019uniqueness}.
  Theorem \ref{Cylindrical improvement} is a generalization of the Neck Improvement Theorem to one dimension higher.
  In the uniform 3-convex setting  we need an additional symmetry improvement (Theorem \ref{bowl x R improvement}) that works for 
  the model Bowl$\times\mathbb{R}$.
  The idea originates from the Theorem 5.4 in \cite{brendle2019uniqueness} and  Theorem 2.6 in \cite{angenent2020uniqueness}, the difficulty comes from the non-cylindrical boundary of the neighborhood that looks like Bowl$\times\mathbb{R}$. This will be handled by a barrier that becomes exponentially small on the bad portion of the boundary.

  In section \ref{section proof of main theorem} we first prove the canonical neighborhood Lemmas (Lemma \ref{canonical nbhd lemma}, \ref{blowdownnecklemma}), which assert that away from a compact set, any point has a  large parabolic neighborhood that looks like shrinking $S^1\times\mathbb{R}^2$ or Bowl$\times\mathbb{R}$, provided that the soliton is not uniformly 2-convex.
  Let's briefly mention the procedure and the ingredients of the proof:
  \begin{enumerate}
    \item By the result of White \cite{white2000size, white2003nature} (cf \cite{haslhofer2017mean}) every blow-down of the convex ancient solution is self-similar,
        therefore must be one of the cylinder $S^k\times\mathbb{R}^{3-k}$. Roughly speaking, since we assumed that the solution is not uniformly 2-convex,
        the only nontrivial case is when a blow-down limit is $S^1\times\mathbb{R}^2$. Then all the blow-down limits are $S^1\times\mathbb{R}^2$ up to some rotation using the
        monotonicity formula.
    \item We argue by contradiction, choose a contradicting sequence and pass to a subsequential limit after rescaling. Moreover the limit will contain a line.
        In our case, the limit is also convex and non-collapsed. To achieve this, we need to use the long range curvature estimate and pointwise curvature derivative estimate which was first established by White \cite{white2000size, white2003nature}.
        Roughly speaking, in a mean convex ancient solution one has uniformly bounded curvature at bounded (parabolic) distance after renormalization. Moreover the curvature derivatives are controlled by the curvature.
    \item The smallest principle curvature vanishes in a time slice because there is a line in it.  The maximum principle by Hamilton \cite{hamilton1986four} (cf \cite{white2003nature}, \cite{haslhofer2017mean}) then forces the solution to split off a line.
        Consequently the limit is a product of a line and a convex, uniformly 2-convex, non-collapsed ancient solution.

    \item With the help of the blow down analysis, we can use a geometric argument to rule out the possibility that the limit is a product of a line with a compact ancient solution. To reach a contradiction it suffices to show that the convex, uniformly 2-convex, non-collapsed, noncompact ancient solution is either a cylinder or Bowl soliton, but that's exactly the result of Brendle--Choi \cite{brendle2019uniqueness, brendle2018uniqueness} .
  \end{enumerate}

  Then we proceed to prove the main theorem by combining all the above ingredients.
  The strategy to prove $SO(2)$ symmetry is to find a $SO(2)$ invariant vector field that is tangential to $M$. This part is very similar to Theorem 5.4 of \cite{brendle2019uniqueness}, we will follow closely to that with all the necessary details.

  We point out that section \ref{section symmetry improvment} is independent of section \ref{section proof of main theorem} and therefore can be of its own interest.

     Lastly, we want to mention some subsequent progress on this problem.
  Built on the result of this paper and other techniques,  Choi-Haslhofer-Hershkovits \cite{choi2021classification} obtained
  the complete classification of non-collapsed translators: they must
  be one of the solutions in the family of entire graph translators constructed by \cite{Hoffman2019}. \newline

\textbf{Acknowledgement:} The author would like to thank his advisor Simon Brendle for his helpful discussions and encouragement.

\section{Preliminary}\label{section preliminary}
Mean curvature flow of hypersurface is a family of immersion: $F: M^n\rightarrow\mathbb{R}^{n+1}$ which satisfies
\begin{align*}
  \frac{\partial F}{\partial t} & =\vec{H}
\end{align*}
Denote $M_t=F(M,t)$. By abuse of notation we sometimes identify $M$ with the image $M_0$ (or $M_{-1}$, depending on the context).

A mean curvature solution is called ancient if it exists for $(-\infty,T)$ for some $T$. An eternal solution is the solution which exists for $(-\infty,+\infty)$.

An important special case for ancient or eternal solution is the translating solutions, they are characterized by the equation
\begin{align*}
   H = \left<V,\nu\right>
\end{align*}
for some fixed nonzero vector $V$.  In this paper 
we usually use $\omega_3$ in place of $V$, where $\omega_3$ is a unit vector in the direction of the $x_3$ axis.

The surface $M_t=M_0+tV$ is a solution of the mean curvature flow  provided that $M_0$ satisfies the translator equation.

For a point $x$ on a hypersurface $M^n\subset\mathbb{R}^{n+1}$ and radius $r$, we use $B_r(x)$ to denote the Euclidean ball and use $B_{g}(x,r)$ to denote the geodesic ball with respect to the metric $g$ on $M$ induced by the embedding.
As in the  \cite[page 55]{brendle2019uniqueness} (or \cite[pages 188--190]{huisken2009mean}) for a space-time point $(\bar{x},\bar{t})$ in a mean curvature flow solution, $\hat{\mathcal{P}}(\bar{x},\bar{t},L,T)=B_{g(\bar{t})}(\bar{x},LH^{-1})\times [\bar{t}-TH^{-2},\bar{t}]$ where $H=H(\bar{x},\bar{t})$.
Again by abuse of notation we usually think of the parabolic neighborhood as its image under the embedding $F$ into the space-time $\mathbb{R}^{n+1}\times\mathbb{R}$

In the next we define the meaning of two mean curvature flow solution being close to each other in an intrinsic way:

\begin{Def}\label{closenessdef}
   Suppose that $M_t$, $\Sigma_t$ are mean curvature flow solutions and $x\in M_t$. Let $\mathcal{M}$ and ${\Sigma}$ be their space-time track. We say that the parabolic neighborhood $\mathcal{U}$=$\hat{\mathcal{P}}(x,t,L,T)$ is $\varepsilon$ close to in $C^k$ norm to a piece of $\Sigma$, if there exists $q\in\Sigma_t$ and a map $\varphi$ such that:
  \begin{enumerate}
    \item $\varphi$ is defined on a parabolic neighborhood $\bar{\mathcal{U}}=\hat{\mathcal{P}}(q,t,(1+\sqrt{\varepsilon})L,T)$ of $\Sigma$
    \item $\varphi(\hat{\mathcal{P}}(q,t,(1-\sqrt{\varepsilon})L,T)\subset\hat{\mathcal{P}}(x,t,L,T)
        \subset\varphi(\hat{\mathcal{P}}(q,t,(1+\sqrt{\varepsilon})L,T))$
    \item
       $\|\varphi(\cdot,t)-Id\|_{C^{k}}<\varepsilon$, where $Id$ is the identity map from $\Sigma_t$ to itself, and $C^k$ is computed with respect to the metric on $\Sigma_t$.
  \end{enumerate}
\end{Def}

   \begin{Rem}\label{closenessdef rmk}
   For our purpose we will take $\varphi$ to be the graph map, which means $\varphi(\cdot,t)=Id+w(\cdot,t)\nu$ where $\nu$ is the outward normal of $\Sigma_t$ and $\| w\|_{C^k}<\varepsilon$.
   \end{Rem}

\begin{Rem}
  If $M_t^j$ converge to the limit $\Sigma_t$ and $(x_j,t_j)\rightarrow (x,t)$,  then for any $\varepsilon,L,T$, the parabolic neighborhood $\hat{\mathcal{P}}(x_j,t_j,L,T)$ is $\varepsilon$ close to a piece of $\Sigma$ for sufficiently large $j$.
  If we have convergence in Euclidean space-time, i.e. in $B_R(0)\times[-R^2,0]$ for any $R>0$, this will imply the Cheeger-Gromov convergence. The converse, however,  is not true in general.
\end{Rem}

\section{Symmetry Improvement}\label{section symmetry improvment}
In the section, we prove two local symmetry improvement  for the parabolic neighborhood of a mean curvature flow solution that looks like shrinking cylinder $S^1\times\mathbb{R}^2$ or the translating Bowl$^2\times\mathbb{R}$.
We begin with some definitions.

Following  \cite[Defintion 4.1]{brendle2019uniqueness} we define the normalized rotation vector field and the notion of the symmetry.
We also define the notion of $(\varepsilon,R)$ cylindrical point, which is parallel to the notion of $\varepsilon$-neck in \cite[page 55]{brendle2019uniqueness}.
\begin{Def}
  A vector field $K$ in $\mathbb{R}^4$ is called the normalized rotation vector field if $K(x)=SJS^{-1}(x-q)$. Where $S\in O(4)$ and $J$ is a $4 \times 4$ matrix whose only nonzero entries are $J_{12}=1, J_{21}=-1$. $q\in \mathbb{R}^4$ is arbitrary. For later use, we define another matrix $J'$ whose only nonzero entries are $J'_{34}=1, J'_{43}=-1$. Explicitly,

  \[J=\begin{bmatrix}
    0 & -1 & 0 & 0 \\
    1 & 0 & 0 & 0 \\
    0 & 0 & 0 & 0 \\
    0 & 0 & 0 & 0
  \end{bmatrix}
  \ \ \ J'=\begin{bmatrix}
    0 & 0 & 0 & 0 \\
    0 & 0 & 0 & 0 \\
    0 & 0 & 0 & -1 \\
    0 & 0 & 1 & 0
  \end{bmatrix}\]
\end{Def}

It's clear that a normalized rotation vector field is a rotation and a translation of the standard rotational field $K_0(x)=Jx$ rotating around the plane $x_1=x_2=0$(which is called the rotation plane). The norm of $K$ is equal to the distance away from the rotation plane.

\begin{Def}\label{varepsilonsymmetrydef}
  Let $M_t$ be a  mean curvature flow solution of positive mean curvature.
  A space-time point $(\bar{x},\bar{t})$ is called $\varepsilon$ symmetric if there exist a normalized rotation field $K$ such that $| \left<K,\nu\right>| H\leq \varepsilon$ and $| K| H\leq 5$ in a parabolic neighbourhood $\hat{\mathcal{P}}(\bar{x},\bar{t},100,100^2)$
\end{Def}

\begin{Def}\label{varepsiloncylindricaldef}
  A point $(\bar{x},\bar{t})$ is said to be $(\varepsilon,R)$ cylindrical if the parabolic neighborhood $\hat{\mathcal{P}}(\bar{x}, \bar{t},R^2,R^2)$ is $\varepsilon$ close in $C^{10}$  to a family of shrinking $S^1\times \mathbb{R}^2$ after a parabolic rescaling
  such that $H(\bar{x},\bar{t})=1$.
\end{Def}

We need to control how far 
two normalized rotation vector fields are away from each other when they are both very tangential to a piece of cylinder $S^1\times \mathbb{R}^2$ or a piece of $\text{Bowl}\times \mathbb{R}$. The following Lemma \ref{NRVFDistanceCyl} and \ref{NRVFDistanceBowl} are both inspired by the Lemma 4.2 in \cite{brendle2019uniqueness}.

Before that we  package up several linear algebra facts:

\begin{Lemma}\label{symmetryrigidity}
  If K is a normalized rotation  vector field and $U$ is an open set in $S^1\times\mathbb{R}^2$, $V$ is an open set in $Bowl^2\times \mathbb{R}$. Here the $Bowl^2$ is the bowl soliton which is rotationally symmetric in the $x_1x_2$-plane.
  Let $A\in so(4)$ and $ b\in \mathbb{R}^4$
  Then
  \begin{enumerate}
    \item If $\left<K,\nu\right>=0$ in $U$, then $K(x)=\pm Jx$ or $\pm J'(x-a)$, where $a\in \mathbb{R}^4$ is arbitrary.
    \item If $\left<K,\nu\right>=0$ in $V$ then $K=\pm Jx$
    \item If  $\left<[A,J]x-Jb,\nu\right>=0$ in $U$ then $[A,J]x-Jb\equiv 0$. If in addition $A_{21}=A_{43}=0$, then $A=0, b=0$
    \item If  $\left<[A,J]x-Jb,\nu\right>=0$ in $V$ then $[A,J]x-Jb\equiv 0$.
        If in addition $A_{21}=A_{43}=0$, then $A=0, b=0$
  \end{enumerate}

\end{Lemma}

\begin{Lemma}\label{NRVFDistanceCyl}
  There exist uniform constants $ \ \varepsilon_{cyl}<1/100$ and $C>1$ with the following properties: for any $\varepsilon\leq \varepsilon_{cyl}$,
   suppose that $K^{(1)},K^{(2)}$ are normalized rotation  vector fields,
    $M$ is a hypersurface in $\mathbb{R}^4$ which is $\varepsilon_{cyl}$ close
    (in the $C^{4}$-norm)
    to  $S^1\times \bar{B}_{20}(0)$ , a subset   of $S^1\times \mathbb{R}^2$, where the radius of $S^1$ is 1 and $\bar{B}_{20}(0)$ is the disk in $\mathbb{R}^2$ with radius 20. $\bar{x} \in M$ is a point $2\varepsilon_{cyl}$ close to $S^1\times \{0\}$
    in the sense that there is a $\bar{q}\in S^1\times \{0\} \subset S^1\times B_{20}(0)$ such that $| \bar{x} \bar{q}| \leq 2\varepsilon_0$
    (by $\varepsilon_{cyl}$ closeness such an $\bar{x}$ exists).
    If
  \begin{itemize}
    \item $| \left<K^{(i)},\nu\right>| H\leq \varepsilon$ in $B_{g}(\bar{x},H^{-1}(\bar{x}))\subset M$
    \item $| K^{(i)}| H\leq 5$ in $B_{g}(\bar{x},10H^{-1}(\bar{x}))\subset M$
  \end{itemize}
  for i=1,2, where $g$ denotes the metric on $M$.
  Then

  \begin{align*}
    \min\left\{ \sup_{B_{LH(\bar{x})^{-1}}(\bar{x})}| K^{(1)}-K^{(2)}| H(\bar{x}),\sup_{B_{LH(\bar{x})^{-1}}(\bar{x})}| K^{(1)}+K^{(2)}| H(\bar{x})\right\} \leq C(L+1)\varepsilon
  \end{align*}
  for any $L\geq 1$.

\end{Lemma}

\begin{proof}
   We first prove for $L=10$.
   Argue by contradiction, suppose that the conclusion is not true, then there exist a sequence of pointed hypersurfaces $(M_j,p_j)$ that are $1/j$ close to $S^1\times \bar{B}_{20}(0)\subset S^1\times \mathbb{R}^2$ in $C^{10}$ norm
   and $| p_j-(1,0,..,0)| \leq 1/j$. We denote the metric on $M_j$ induced by the embedding to be $g_j$.

    Moreover we may assume that there exist normalized rotation fields $K^{(i,j)},\ i=1,2,$ and $\varepsilon_j<1/j$ such that
    \begin{itemize}
      \item $| \left<K^{(i,j)},\nu\right>| H\leq \varepsilon_j$ in $B_{g_j}(p_j,{H^{-1}(p_j)})$
      \item $| K^{(i,j)}| H\leq 5$  in $B_{g_j}(p_j,10H^{-1}(p_j))$
    \end{itemize}
    for $i=1,2$
    \begin{itemize}
      \item  \begin{align*}
                \min\left\{
                \begin{array}{c}
                  \sup_{B_{10H(p_j)^{-1}}(p_j)}| K^{(1,j)}-K^{(2,j)}| H(p_j) \\
                  \sup_{B_{10H(p_j)^{-1}}(p_j)}| K^{(1,j)}+K^{(2,j)}| H(p_j)
                \end{array}
                \right\}\geq j\varepsilon_j
             \end{align*}
    \end{itemize}

   Therefore $M_j$ converges to $M_{\infty}=S^1\times \bar{B}_{20}(0)\subset S^1\times \mathbb{R}^2$ whose metric is denoted by $g_{\infty}$.
    Moreover $p_j\rightarrow p_{\infty}=(1,0,..,0)$ and $H(p_j)\rightarrow 1$.

   Suppose that $K^{(i,j)}(x)=S_{(i,j)}JS_{(i,j)}^{-1}(x-b_{(i,j)})$ where $S_{(i,j)}\in O(4)$ and $b_{(i,j)}\in\mathbb{R}^4$.
   Without loss of generality , we may assume that $b_{(i,j)}$ {is} orthogonal to the kernel of the matrix $JS_{(i,j)}^{-1}$. This way $| S_{(i,j)}JS_{(i,j)}^{-1}b_{(i,j)}| =| b_{(i,j)}| $.

   Now we have $| S_{(i,j)}JS_{(i,j)}^{-1}(p_j-b_{(i,j)})| \leq5H(p_j)^{-1}<5+5/j$, therefore
   \[| b_{(i,j)}| =| S_{(i,j)}JS_{(i,j)}^{-1}b_{(i,j)}| <5+5/j+| p_j| <6+6/j<10\] for large $j$.
   Passing to a further subsequence such that $S_{(i,j)}\rightarrow S_{(i,\infty)}$ and $b_{(i,j)}\rightarrow b_{(i,\infty)}$.
   Thus $K^{(i,j)}\rightarrow K^{(i,\infty)}$ locally smoothly, where $K^{(i,\infty)}(x)=S_{(i,\infty)}JS_{i,\infty}^{-1}(x-b_{(i,\infty)})$.

   The convergence also implies that $\left<K^{(i,\infty)},\nu\right>=0$ in $B_{g_{\infty}}(p_{\infty},1)$ and $| K^{(i,\infty)}| \leq 5$ in $B_{g_{\infty}}(p_{\infty},10)$ (we removed $H$ because $H\equiv 1$ on $M_{\infty}$). By Lemma \ref{symmetryrigidity}, $K^{(i,\infty)}(x)=\pm Jx$ or $\pm J'(x-a)$.

   If $K^{(i,\infty)}=\pm J'(x-a)$, then $\sup_{B_{g_{\infty}}(p_{\infty},10)}| K^{(i,\infty)}|
   \geq 10+| a^T-p_{\infty}^T| \geq10$ (the superscript $^T$ means projection onto $\mathbb{R}^2$), this is impossible. So $K^{(i,\infty)}=\pm Jx$.

   Without loss of generality we may assume $K^{(1,\infty)}=K^{(2,\infty)}=Jx$, for the other cases we simply flip the sign and the same argument applies. Hence $S_{(i,\infty)}JS_{(i,\infty)}^{-1}=J$ and $S_{(i,\infty)}JS_{(i,\infty)}^{-1}b_{(i,\infty)}=0$. Since $b_{(i,\infty)}$ is orthogonal to the kernel of $JS_{(i,\infty)}^{-1}$, we know that $b_{(i,\infty)}=0$. In particular we know that $b_{(i,j)}\rightarrow 0$.

   We may also assume that $S_{(i,\infty)}= Id$. For otherwise we can replace $S_{(i,j)}$ by $\tilde{S}_{(i,j)}=S_{(i,j)}S_{(i,\infty)}^{-1}$, then   $\tilde{S}_{(i,j)}\rightarrow Id$.
   It's straightforward to check that $\tilde{S}_{(i,j)}J\tilde{S}_{(i,j)}^{-1}=S_{(i,j)}JS_{(i,j)}^{-1}$, so such a replacement doesn't change $K^{(i,j)}$.

    Consequently $S_{(1,j)}^{-1}S_{(2,j)}$ is close to $Id$. Since exp is a local diffeomorphism between $so(4)$ and $SO(4)$ near the origin and the Id respectively,  we can write $S_j=S_{(1,j)}^{-1}S_{(2,j)}=\exp(A_j)$ with $A_j\in so(4)$ being small.
    Thus $S_j=Id+A_j+O(| A_j| ^2)$. Also we have $S_j^{-1}=\text{exp}(-A_j)=Id-A_j+O(|
    A_j| ^2)$.

    We may assume that $(A_j)_{21}=(A_j)_{43}=0$. If not, we can replace $S_{(2,j)}$ by $\tilde{S}_{(2,j)}=S_{(2,j)}\text{exp}(-\eta_jJ-\theta_j J')$.  This doesn't change $K^{(2,j)}$, since the matrix $\text{exp}(-\eta_jJ-\theta_j J')\in SO(4)$ commutes with $J$.
    Then $S_j$ becomes $\tilde{S}_j=\exp(A_j)\exp(-\eta_jJ-\theta_jJ')$. Suppose that $\eta_j$ and $\theta_j$ are small, then the local diffeomorphism property of exp and Baker-Campbell-Hausdorff formula implies that there is a unique small matrix $\tilde{A}\in so(4)$ such that $\tilde{S}_j=\exp(\tilde{A_j})$ and $\tilde{A}_j=A_j-\eta_jJ-\theta_jJ'+O([A_j,\eta_jJ+\theta_jJ'])$ is a smooth function of $A_j, \eta_j, \theta_j$.
    Note that $A_j=0, \eta_j=\theta_j=0 \Rightarrow \tilde{A}_j=0$ , also $[A_j,\eta_jJ+\theta_jJ']=o(| A_j|  | \eta_j| +| A_j|  | \theta_j| )$. So
    $(\tilde{A}_j)_{21}=(A_j)_{21}-\eta_j+o(| A_j|  | \eta_j| +| A_j|  | \theta_j| )$ and
      $(\tilde{A}_j)_{43}=(A_j)_{43}-\theta_j+o(| A_j| +| \theta_j| )$.
    Now Implicit Function Theorem gives that, for all small enough $A_j$ there exists small $\eta_j, \theta_j$ near $(A_j)_{21}, (A_j)_{43}$ such that $(\tilde{A}_j)_{21}=(\tilde{A}_j)_{43}=0$.

   With $A_{21}=A_{43}=0$, a direct computation gives that $| [A_j,J]x| =| A_jx| $ for any $x\in\mathbb{R}^4$.

   Define \[W^j=\frac{K^{(1,j)}-K^{(2,j)}}{\sup\limits_{B_{10H(p_j)^{-1}}(p_j)}| K^{(1,j)}-K^{(2,j)}| }\]

   Recall that $S_{(i,j)}=Id+o(1)$ and $S_j=S_{(1,j)}^{-1}S_{(2,j)}$. For simplicity we write $P_j=S_{(1,j)}JS_{(1,j)}^{-1}-S_{(2,j)}JS_{(2,j)}^{-1}$ and $c_j=-S_{(1,j)}JS_{(1,j)}^{-1}(b_{(1,j)}-b_{(2,j)})$. Then
   \begin{align}\label{Kdiff}
     K^{(1,j)}(x)-K^{(2,j)}(x)=&P_j(x-b_{(2,j)})+c_j
   \end{align}

   Now we compute $P_j$:
   \begin{align}
     P_j= & S_{(1,j)}(J-S_jJS_j^{-1})S_{(1,j)}^{-1}\\
     =&S_{(1,j)}(J-\exp(A_j)J\exp(-A_j))S_{(1,j)}^{-1} \notag\\
     =&(Id+o(1))(-[A_j,J]+o(| A_j| ))(Id+(1)) \notag\\
     =& -[A_j,J]+o(| A_j| )  \notag
   \end{align}
    so $|
    A_j| \leq C| [A_j,J]| \leq C(| P_j| +o(| A_j| ))$, absorbing $o(| A_j| )$ in  the left to get
    \begin{align}\label{A less than P}
      | A_j| \leq C| P_j|
    \end{align}
    Let $\sup\limits_{B_{10H(p_j)^{-1}}(p_j)}| K^{(1,j)}-K^{(2,j)}|  = Q_j$.

    Since $B_1(b_{(2,j)})\subset B_{10H(p_j)^{-1}}(p_j)$, we can put $x=b_{(2,j)}$ in  (\ref{Kdiff}) to get that
    \begin{align}\label{c less than Q}
      | c_j| \leq Q_j
    \end{align}
    Moreover we have:
    \begin{align}\label{P < Q}
          | P_j| \leq& C\sup\limits_{B_1(b_{(2,j)})}| P_j(x-b_{(2,j)})| \\
          \leq  & C\sup\limits_{B_{10H(p_j)^{-1}}(p_j)}| K^{(1,j)}-K^{(2,j)}| +C| c_j|  \leq  CQ_j \notag
    \end{align}

    By (\ref{A less than P}), (\ref{P < Q}) we have
    \begin{align}\label{A < Q}
      | A_j| \leq C| P_j| \leq CQ_j
    \end{align}

    Now with (\ref{c less than Q}) (\ref{A < Q}) we have that $| A_j| /Q_j$ and $| c_j| /Q_j$ are uniformly bounded.
    It's then possible to pass to a subsequence to get $A_j/Q_j\rightarrow A$ and $c_j/Q_j\rightarrow c$ for some $A\in so(4)$ with $A_{21}=A_{43}=0$ and $c\in\mathbb{R}^4$.

    Note that $c_j/Q_j\in \text{Im}(S_{(1,j)}J)=\ker(J'S_{(1,j)}^{-1})$, so we have

    \begin{align}\label{}
      J'c=\lim\limits_{j\rightarrow \infty} J'\left(\frac{c_j}{Q_j}\right) & =\lim\limits_{j\rightarrow \infty} (J'-J'S_{(1,j)}^{-1})\left(\frac{c_j}{Q_j}\right)=0\cdot c=0
    \end{align}
    This implies that   $c\in \ker(J')=\text{Im}(J)$, therefore we can write $c=-Jb$.

    To summarize, we have the following:
    \begin{align}\label{}
      W^j(x)=\frac{(A_j+o(| A_j| ))(x-b_{(2,j)})+c_j}{Q_j}\rightarrow [A,J]x-Jb := W^{\infty}(x)
    \end{align}
     uniformly in any compact set.

    Note that $\sup_{B_{10H(p_j)^{-1}}(p_j)}| W^j| =1$, thus $W^{\infty}$ can not be identically 0.

   Since
    \begin{align}\label{}
      \left\{
    \begin{array}{c}
      | \left< K^{(1,j)}-K^{(2,j)},\nu \right>|  \leq 2H^{-1}\varepsilon_j \text{ in } B_{g_j}(p_j,H(p_j)^{-1}) \\
      \sup_{B_{10H(p_j)^{-1}}(p_j)}| K^{(1,j)}-K^{(2,j)}| \geq jH^{-1}\varepsilon_j
    \end{array}\right.
    \end{align}
    we know that  $|\left<W^j, \nu \right>|\leq 2/j\rightarrow 0$ in $B_{g_j}(p_j,H(p_j)^{-1})$

    Taking limit we have  $\left<W^{\infty},\nu\right>=0$ in $B_{g_{\infty}}(p_{\infty},1)$. By Lemma \ref{symmetryrigidity}, $ W^{\infty}\equiv 0$, a contradiction.
    So the result is proved for $L=10$.

    For general $L$, notice that $(K^{(1)}-K^{(2)})(x)$ is in the form of $Ax+b$ for some fixed matrix $A$ and $b\in \mathbb{R}^4$.

   Since $\sup_{B_{10H(\bar{x})^{-1}}(\bar{x})}| K^{(1)}-K^{(2)}| H(\bar{x})\leq C\varepsilon$, we have $| A\bar{x}+b| \leq CH(\bar{x})^{-1}\varepsilon$ and
   $\sup_{| x| \leq 1}| Ax| \leq C\varepsilon H(\bar{x})^{-1}/10$.
   Thus
   \begin{align}\label{}
     \sup_{B_{LH(\bar{x})^{-1}}(\bar{x})}| K^{(1)}-K^{(2)}| H(\bar{x})\leq \left(| A\bar{x}+b| + \sup_{| y| \leq LH(\bar{x})^{-1}}Ay\right)H(\bar{x})\leq C\varepsilon L/10+C\varepsilon
   \end{align}

    The Lemma is proved.
\end{proof}

\begin{Lemma}\label{NRVFDistanceBowl}
   Given $\delta<1/2$, there exists $ \varepsilon_{bowl}$ and $C>1$  depending on $\delta$  with the following properties: let $\Sigma\subset\mathbb{R}^3$ be the Bowl soliton with maximal mean curvature $1$ and $M\subset\mathbb{R}^4$ be a hypersurface with metric $g$.
    Suppose that $ \varepsilon\leq \varepsilon_{bowl}$, \
     $ q\in  \Sigma\times\mathbb{R}$ , and suppose that   $M$ is a graph over the geodesic ball in $\Sigma\times\mathbb{R}$ of radius $2H(q)^{-1}$ centered at $q$. After rescaling by $H(q)^{-1}$,
     the graph $C^4$ norm is $\leq \varepsilon_{bowl}$.
     Let $\bar{x}\in M$ be the point that has rescaled distance $\leq \varepsilon_{bowl}$ to $q$. Suppose that $K^{1},K^{2} $ are normalized rotation vector fields.  If
  \begin{itemize}
    \item $\lambda_1+\lambda_2\geq \delta H$ in $B_g(\bar{x},H(\bar{x})^{-1})$ where $\lambda_1, \lambda_2$ are the lowest two principal curvatures.
    \item $| \left<K^{i}, \nu \right>| H\leq \varepsilon$ in $B_g(\bar{x},H(\bar{x})^{-1})$  for $i=1,2$
    \item $| K^{i}| H\leq 5$ at $\bar{x}$ for $i=1,2$
  \end{itemize}
  Then

  \begin{align*}
    \min\left\{ \sup_{B_{LH(\bar{x})^{-1}}(\bar{x})}| K^{(1)}-K^{(2)}| H(\bar{x}),\sup_{B_{LH(\bar{x})^{-1}}(\bar{x})}| K^{(1)}+K^{(2)}| H(\bar{x})\right\} \leq C(L+1)\varepsilon
  \end{align*}
  for any $L\geq 1$.

\end{Lemma}
\begin{proof}

  First we assume that the tip of $\Sigma$ is the origin and the rotation axis  is $x_3$, also it enclose the positive part of $x_3$ axis.
  Argue by contradiction.
  Suppose that the conclusion is not true, then there exists a sequence of
   points $\bar{q}_j\in \Sigma\times\mathbb{R}$, scaling factors $\kappa_j$ such that $q_j = \kappa_j^{-1}\bar{q}_j $ has mean curvature 1 in hypersurface $\kappa_j^{-1}\Sigma\times\mathbb{R}$ (thus $\kappa_j = H_{\Sigma\times\mathbb{R}}(q_j)$) and a sequence of
  pointed hypersurfaces $(M_j,p_j)$ that are $1/j$ close to the geodesic ball $B_{\tilde{g}_j}(q_j,2)$ in $\kappa_j^{-1}\Sigma\times\mathbb{R}$, where $\tilde{g}_j$ denotes the metric on $\kappa_j^{-1}\Sigma\times\mathbb{R}$.
  Moreover $| p_j-q_j| \leq 1/j$. Without loss of generality we may also assume that $\left<q_j,\omega_4\right>=0$ where $\omega_4$ is a unit vector in the $\mathbb{R}$ direction.

  Further, there exists normalized rotation vector fields $K^{(i,j)}$, $i=1,2$ and $\varepsilon_j<1/j$  such that
   \begin{itemize}
     \item $|
     \left<K^{(i,j)},\nu\right>| H\leq \varepsilon_j$ in $B_{g_j}(p_j,H(p_j)^{-1})\subset M_j$
     \item  $| K^{(i,j)}| H\leq 5$ at $p_j$
   \end{itemize}
   for  $i=1,2$,
   \begin{itemize}
     \item \begin{align*}
             \min\left\{
             \begin{array}{c}
               \sup_{B_{10H(p_j)^{-1}}(p_j)}| K^{(1,j)}-K^{(2,j)}| H(p_j) \\
               \sup_{B_{10H(p_j)^{-1}}(p_j)}| K^{(1,j)}+K^{(2,j)}| H(p_j)
             \end{array}
             \right\}\geq j\varepsilon_j \\
           \end{align*}
   \end{itemize}

  Now the maximal mean curvature of $\kappa_j^{-1}\Sigma\times\mathbb{R}$ is $\kappa_j$.
  For any $j>2/\delta$,
  by condition $\lambda_1+\lambda_2\geq \delta H$ and approximation we know that $\frac{\lambda_1+\lambda_2}{H}\geq \delta-1/j\geq\frac{\delta}{2}$ around $q_j$.
  The asymptotic behaviour of the Bowl soliton indicates that $\frac{H(q_j)}{\kappa_j}<C(\delta)$ and $\kappa_j|
  q_j-\left<q_j,\omega_4\right>\omega_4| <C(\delta)$,
  thus $\kappa_j>C(\delta)^{-1}$ and $| q_j| =| q_j-\left<q_j,\omega_4\right>\omega_4| <C(\delta)$.

  We can then pass to a subsequence such that $q_j\rightarrow q_{\infty}$ and $\kappa_j\rightarrow\kappa_{\infty}>C(\delta)^{-1}>0$.
  Consequently $\kappa_j^{-1}\Sigma\times\mathbb{R}\rightarrow \kappa_{\infty}^{-1}\Sigma\times\mathbb{R}$ and  $B_{\tilde{g}}(q_j,2)\rightarrow B_{\tilde{g}_{\infty}}(q_{\infty},2)$ smoothly,
   where $B_{\tilde{g}_{\infty}}(q_{\infty},2)$ is the geodesic ball in $\kappa_{\infty}^{-1}\Sigma\times\mathbb{R}$.

  Combing with the assumption that $(M_j,p_j)$ is $1/j$  close to $(B_{\tilde{g}_j}(q_j,2),q_j)$ and $H(q_j)=1$, we have $M_j\rightarrow B_{\tilde{g}_{\infty}}(q_{\infty},2)$ with $p_j\rightarrow q_{\infty}$
  and $H(q_{\infty})=1$.

    We can write $K^{(i,j)}(x)=S_{(i,j)}JS_{(i,j)}^{-1}(x-b_{(i,j)})$ and assume that $(b_{(i,j)}-p_j)\perp \ker JS_{(i,j)}^{-1}$. Then
    $| b_{(i,j)}-p_j|
    =| S_{(i,j)}JS_{(i,j)}^{-1}(p_j-b_{(i,j)})| =| K^{(i,j)}(p_j)| \leq 5H(p_j)^{-1}\leq 5+10/j$ for large $j$. Then we can pass to a subsequence such that $S_{(i,j)}$ and $b_{(i,j)}$ converge to $S_{(i,\infty)}, b_{(i,\infty)}$ respectively.
    Consequently $K^{i,j}\rightarrow K^{(i,\infty)}$, the limit $K^{(i,\infty)}(x)=S_{(i,\infty)}JS_{(i,\infty)}^{-1}(x-b_{(i,\infty)})$ where $(b_{(i,\infty)}-p_{\infty})\perp \ker J$.

    Next, by the  hypersurface convergence and $| \left<K^{(i,j)},\nu\right>| \rightarrow 0$ we have $\left< K^{(i,\infty)},\nu \right>=0$ on $B_{\tilde{g}_{\infty}}(q_{\infty},2)$ .
  The  Lemma \ref{symmetryrigidity} then implies that $K^{(i,\infty)}=\pm Jx$.

    We only consider the case that $K^{(1,\infty)}=K^{(2,\infty)}= Jx$. The other cases follow by flipping the sign.

    Arguing as in Lemma \ref{NRVFDistanceCyl}, we may assume that $S_{(i,\infty)}=Id$. Now we have $Jb_{(i,\infty)}=0$ and $(b_{(i,\infty)}-p_{\infty})\perp\ker J$. Consequently, $b_{(i,\infty)}=p_{\infty}-Jp_{\infty}$, which is the orthogonal projection of $p_{\infty}$ onto $\ker J$.

     Recall that $p_{\infty}=q_{\infty}\in \kappa_{\infty}^{-1}\Sigma\times\mathbb{R}$ where the splitting direction is $x_4$ and the rotation axis of $\Sigma$ is $x_3$.
     Moreover $H(p_{\infty})=1$.  The structure of the Bowl soliton (see eg Appendix \ref{ode of bowl}) ensures that $Jp_{\infty}<2H(p_{\infty})^{-1}=2$, that means $B_{10H(p_j)^{-1}}(p_j)$ contains $B_1(b_{(i,j)})$.

    Now we can use exactly the same argument as in Lemma \ref{NRVFDistanceCyl} to reach an contradiction for $L=10$, hence the lemma is proved for $L=10$. Then the fact that $K^{(1)}-K^{(2)}$ is an affine vector field gives the result for all $L>1$.

 \end{proof}

The next step is to prove the symmetry improvement. Theorem \ref{Cylindrical improvement} will be used later to handle the neighborhood modeled on shrinking $S^1\times\mathbb{R}^2$, this is parallel to the Theorem 4.4 (Neck Improvement Theorem) in \cite{brendle2019uniqueness}.

\begin{Th}\label{Cylindrical improvement}
There exists  constant  $L_0>1$ and  $0<\varepsilon_0<1/10$ with the following properties: suppose that $M_t$ is  a mean curvature flow solution, if every point in the parabolic neighborhood $\hat{\mathcal{P}}(\bar{x},\bar{t},L_0,L_0^2)$ is $\varepsilon$ symmetric and
$(\varepsilon_0,100)$ cylindrical,
 where $0<\varepsilon\leq \varepsilon_0$, then $(\bar{x},\bar{t})$ is $\frac{\varepsilon}{2}$ symmetric.
\end{Th}
\begin{proof}
   We assume that $L_0$ is large and $\varepsilon_0$ is small depending on $L_0$. The constant $C$ without specified dependence means a universal constant.

  For any space-time point $(y,s)\in \hat{\mathcal{P}}(\bar{x},\bar{t},L_0,L_0^2)$, by assumption there is a normalized rotation field $K^{(y,s)}$ such that
   $| \left<K^{(y,s)},\nu\right>| H\leq \varepsilon$ and $| K^{(y,s)}| H\leq 5$ in a parabolic neighbourhood $\hat{\mathcal{P}}(y,s,100,100^2)$.

   Without loss of generality we may assume $\bar{t}=-1$, $H(\bar{x})=\frac{1}{\sqrt{2}}$ and $| \bar{x}-(\sqrt{2},0,0,0)| \leq \varepsilon_0$. Let $K_0=Jx$ and $\bar{K}=K^{(\bar{x},-1)}$.

   By the cylindrical assumption,
    for each $(y,s)\in\hat{\mathcal{P}}(\bar{x},\bar{t},L_0,L_0^2)$,
     the parabolic neighborhood $\hat{\mathcal{P}}(y,s,100^2,100^2)$ is $C(L_0)\varepsilon_0$ close to the shrinking cylinder
   $S^1_{\sqrt{-2t}}\times \mathbb{R}^2$ in $C^{10}$ norm.

    We  use the polar coordinate $(r\cos\theta,r\sin\theta,z_1,z_2)$:
    $M_t$ is then a radial graph over $z_1z_2$-plane.
    More precisely:
    the set \[\left\{(r\cos\theta,r\sin\theta,z_1,z_2)\big\lvert
    r=r(\theta,z_1,z_2), z_1^2+z_2^2\leq \frac{L_0^2}{2},  \theta\in [0,2\pi]\right\}\] is contained in $M_t$. \\

   \noindent\textbf{Step 1:} Given any $(y,s)$ and $(y', s')$ in $\hat{\mathcal{P}}(\bar{x},-1,L_0,L_0^2)$, we show that $| K^{(y,s)}| H<5$ in $\hat{\mathcal{P}}(y',s',100,100^2)$. Also without loss of generality we can choose $\bar{K}=K_0$.

   Since $\left<K_0,\nu\right>=0$ and $| K_0| H=1$ on $S^1_{\sqrt{-2t}}\times\mathbb{R}$, by approximation we know that $|
   \left<K_0,\nu\right>| H<C(L_0)\varepsilon_0$ and $| K_0| H<1+C(L_0)\varepsilon_0$ in  parabolic neighborhood $\hat{\mathcal{P}}(y,s,100,100^2)$.

    Taking $\varepsilon_0$ small enough such that $\varepsilon_0C(L_0)<\varepsilon_{cyl}$. Then we can apply Lemma \ref{NRVFDistanceCyl} to obtain:
   \begin{align}\label{Kdiffrough}
    \min\left\{ \sup_{B_{10^3L_0}(\bar{x})}| K_0-K^{(y,s)}| H(\bar{x}),\sup_{B_{10^3L_0}(\bar{x})}| K_0+K^{(y,s)}| H(\bar{x})\right\} \leq CL_0\varepsilon_0
  \end{align}
  By the structure of the shrinking cylinder and approximation, each time slice of $\hat{\mathcal{P}}(y',s',100,100^2)$ is contained in $B_{10^3L_0}(\bar{x})$, then we have $| K^{(y,s)}| H<1+C(L_0)\varepsilon_0<5$ in $\hat{\mathcal{P}}(y',s',100,100^2)$ by taking $\varepsilon_0$ small enough.

   Applying (\ref{Kdiffrough})  to $\bar{K}$, then without loss of generality we assume that $\sup_{B_{10^3L_0}(\bar{x})}| K_0-\bar{K}|
   H(\bar{x})<CL_0\varepsilon_0$. This  means that $\bar{K}=SJS^{-1}(x-q)$ with some $S\in SO(4)$ and $q\in\mathbb{R}^4$ such that $| S-Id| +| q| <C(L_0)\varepsilon_0$.
    Now we rotate $M_t$ by $S$, then $S^1_{\sqrt{-2t}}\times \mathbb{R}^2$ is still a $C(L_0)\varepsilon_0$ approximation of $M_t$, but $\bar{K}$ becomes $Jx$.\\

   \noindent\textbf{Step 2:} We show that, for each point $(y,s)\in \hat{\mathcal{P}}(\bar{x},\bar{t},L_0,L_0^2)$, there are constants $a_i,b_i$, $i=0,1,2$ with  $| a_i|
   +| b_i| \leq C(L_0)\varepsilon$ such that
   \begin{align}\label{finetuneK}
      | \left<\bar{K},\nu\right>-(a_0+a_1z_1+a_2z_2)\cos\theta -(b_0+b_1z_1+b_2z_2)\sin\theta| \leq C(L_0)\varepsilon\varepsilon_0+C\varepsilon(-t)^{1/2}
   \end{align}
   in $\hat{\mathcal{P}}(y,s,100,100^2)$

   First, we can find a sequence of points
    $(x_i,t_i)\in \hat{\mathcal{P}}(\bar{x},\bar{t},L_0,L_0^2)$, $i=0,1,...,N$ with $N\leq C(L_0)$ such that $(x_0,t_0)=(\bar{x},\bar{t})$, $(x_N,t_N)=(y,s)$ and
     $(x_{i+1},t_{i+1})\in\hat{\mathcal{P}}(x_i,t_i,1,1)$.
     By the structure of the shrinking cylinder and approximation, the time $t_{i+1}$ slice of $\hat{\mathcal{P}}(x_i,t_i,100,100^2) $ contains $ B_{g_{t_{i+1}}}(x_{i+1},{H(x_{i+1})^{-1}}) $.
     Now we are ready to apply Lemma \ref{NRVFDistanceCyl} to obtain that
     \begin{align*}
    \min\left\{ \sup_{B_{10^3L_0}(\bar{x})}| K^{(x_i,t_i)}-K^{(x_{i+1},t_{i+1})}| ,\sup_{B_{10^3L_0}(\bar{x})}| K^{(x_i,t_i)}+K^{(x_{i+1},t_{i+1})}| \right\} \leq C(L_0)\varepsilon
  \end{align*}
   Summing up the above inequality for $i=0,1,...,N-1$ we get
   \begin{align*}
   \min\left\{ \sup_{B_{10^3L_0}(\bar{x})}|
   \bar{K}-K^{(y,s)}| ,\sup_{B_{10^3L_0}(\bar{x})}| \bar{K}+K^{(y,s)}| \right\} \leq C(L_0)\varepsilon
  \end{align*}
  Note that each time slice of $\hat{\mathcal{P}}(y,s,100,100^2)$ is contained in $B_{10^3L_0}(\bar{x})$. So without loss of generality we may assume that
  \[| \bar{K}-K^{(y,s)}| \leq C(L_0)\varepsilon\] in $\hat{\mathcal{P}}(y,s,100,100^2)$.

 Then we can choose a matrix $S\in SO(4)$ and $q\in \mathbb{R}^4$ with $| S-Id| +| q| \leq C(L_0)\varepsilon$ such that
     $K^{(y,s)}=SJS^{-1}(x-q)$.
     By a simple computation we will find the constant $a_i,b_i$, $i=0,1,2$ such that

      $\left<\bar{K}-K^{(y,s)},\nu \right>=(a_0+a_1z_1+a_2z_2)\cos\theta+(b_0+b_1z_1+b_2z_2)\sin\theta$ \ \
      on $S^1_{\sqrt{-2t}}\times \mathbb{R}^2$
      and $| a_i| ,| b_i| $ are bounded by $C(L_0)\varepsilon$ .

     By approximation we have
     \[| \left<\bar{K}-K^{(y,s)},\nu \right> -(a_0+a_1z_1+a_2z_2)\cos\theta-(b_0+b_1z_1+b_2z_2)\sin\theta| \leq C(L_0)\varepsilon\varepsilon_0\] in
     $\hat{\mathcal{P}}(y,s,100,100^2)$.
     In fact, since $\bar{K}-K^{(y,s)}$ changes linearly at a rate $\leq C(L_0)\varepsilon$, in the  approximation process the $\bar{K}-K^{(y,s)} $ term will produce an error of at most $C(L_0)\varepsilon_0\varepsilon$ , while $\nu$ will have an error up to $C(L_0)\varepsilon_0$. Combining them up we have the desired the error term.

     Since $| \left<K^{(y,s)},\nu\right>| \leq C\varepsilon H^{-1}<C\varepsilon(-t)^{1/2}$, we eventually obtain:
     \begin{align*}
      | \left<\bar{K},\nu\right>-(a_0+a_1z_1+a_2z_2)\cos\theta -(b_0+b_1z_1+b_2z_2)\sin\theta| \leq C(L_0)\varepsilon\varepsilon_0+C\varepsilon(-t)^{1/2}
   \end{align*}
    in
     $\hat{\mathcal{P}}(y,s,100,100^2)$.

     Note that $a_i, b_i$  depend on the choice of $(y,s)$.

     An easy consequence is that $| \left<\bar{K},\nu\right>| \leq C(L_0)\varepsilon$ in $\hat{\mathcal{P}}(\bar{x},-1,L_0,L_0^2)$.\\

     \noindent\textbf{Step 3:}
     Set $u=\left<\bar{K},\nu\right>$, we have the following equation in the polar coordinate:
     \begin{align}\label{heateqnforu}
       | \frac{\partial u}{\partial t}-(\frac{\partial^2 u}{\partial z_1^2}+\frac{\partial^2 u}{\partial z_2^2}+\frac{1}{-2t}\frac{\partial^2 u}{\partial \theta^2}+\frac{1}{-2t}u)| \leq C(L_0)\varepsilon\varepsilon_0
     \end{align}
     for $z_1^2+z_2^2\leq \frac{L_0^2}{4}$ and $-\frac{L_0^2}{4}\leq t\leq -1$

     First,  $u$ satisfies the parabolic Jacobi equation
     \begin{align}\label{jacobieqn}
       \partial_t u=\Delta_{M_t} u+| A| ^2 u
     \end{align}
     in $\hat{\mathcal{P}}(\bar{x},-1,L_0,L_0^2)$.

     On $S^1_{\sqrt{-2t}}\times \mathbb{R}^2$,  this equation is  (\ref{heateqnforu}) with right hand side being exactly 0.

     Now we work on $\left\{z_1^2+z_2^2\leq \frac{L_0^2}{2}, -\frac{L_0^2}{2}\leq t\leq -1, , 0\leq\theta\leq2\pi\right\}$ , which is contained in $\hat{\mathcal{P}}(\bar{x},-1,L_0,L_0^2)$ under polar coordinate.
     By approximation the coefficient error is $C(L_0)\varepsilon_0$.
     Step 1 gives that $| u| \leq C(L_0)\varepsilon$, then the parabolic interior  gradient estimate gives that $| \nabla u|
     +| \nabla^2 u| \leq C(L_0)\varepsilon$ for $z_1^2+z_2^2\leq \frac{L_0^2}{4}$ and $-\frac{L_0^2}{4}\leq t \leq -1$.

     Then
     \begin{align*}
       &| \frac{\partial u}{\partial t}-\frac{\partial^2 u}{\partial z_1^2}-\frac{\partial^2 u}{\partial z_2^2}-\frac{1}{-2t}\frac{\partial^2 u}{\partial \theta^2}-\frac{1}{-2t}u| \\
       &\leq C(L_0)\varepsilon_0| D^2u|
       +C(L_0)\varepsilon_0| u|  \\
       &\leq C(L_0)\varepsilon_0\varepsilon
     \end{align*}
      which justifies (\ref{heateqnforu}) \\

      \noindent\textbf{Step 4:}

      We deal with the (nonhomogenous) linear equation (\ref{heateqnforu}).

      Define

      $\Omega_l(\bar{z}_1,\bar{z}_2) =\{(z_1,z_2)\in\mathbb{R}^2|  \ | z_i-\bar{z}_i|
      \leq l\}$

      $\Omega_l=\Omega_l(0, 0)$

       $\Gamma_l=\{(z_1,z_2)\in \Omega_l,\theta\in [0,2\pi],t\in [-l^2,-1] \big\lvert  t=-l^2 \text{ or } | z_1| =l \text{ or } | z_2| =l\}$

      Let $\tilde{u}$ solves
       \begin{align*}
         \frac{\partial \tilde{u}}{\partial t}-\frac{\partial^2 \tilde{u}}{\partial z_1^2}-\frac{\partial^2 \tilde{u}}{\partial z_2^2}-\frac{1}{-2t}\frac{\partial^2 \tilde{u}}{\partial \theta^2}-\frac{1}{-2t}\tilde{u}=0
       \end{align*}
       in $\Omega_{L_0/4}\times [0,2\pi]\times [-L_0^2/16,-1]$ and satisfies the boundary condition $\tilde{u}=u$ on $\Gamma_{L_0/4}$.
       By (\ref{heateqnforu}) and maximum principle
       \begin{align}\label{diff tildeu and u}
         | u-\tilde{u}| \leq C(L_0)\varepsilon\varepsilon_0
       \end{align}

       Next, we will compute Fourier coefficients of $\tilde{u}$  and analysis each of them.  Let
       \begin{align*}
         \tilde{u}_m(z,t)=\frac{1}{\pi}\int_{0}^{2\pi}\tilde{u}(z,t,\theta)\cos(m\theta)d\theta, \ \  \tilde{v}_m(z,t)=\frac{1}{\pi}\int_{0}^{2\pi}\tilde{v}(z,t,\theta)\sin(m\theta)d\theta \\
       \end{align*}
        where $z=(z_1,z_2)$.

       For any $(\bar{z}_1,\bar{z}_2,\theta_0,t_0)\in\Omega_{L_0/4}\times [0,2\pi]\times [-L_0^2/16,-1]$,
        by (\ref{finetuneK}) and (\ref{diff tildeu and u}), we have:
        \begin{align}\label{refinedtuneK}
          | \tilde{u}-(a_0+a_1z_1+a_2z_2)\cos\theta -(b_0+b_1z_1+b_2z_2)\sin\theta| \leq C(L_0)\varepsilon\varepsilon_0+C\varepsilon(-t)^{1/2}
        \end{align} in
       $\Omega_{(-t_0)^{\frac{1}{2}}}(\bar{z}_1,\bar{z}_2)\times[0,2\pi]\times[2t_0,t_0]$.

       In particular, for all $m\geq 2$,
       we have
       \begin{align*}
         | \tilde{u}_m| +| \tilde{v}_m| \leq (C(L_0)\varepsilon\varepsilon_0+C\varepsilon)(-t)^{1/2}
       \end{align*}
        in $\Omega_{L_0/4}\times[-L_0^2/16,-1]$.

       Let $\hat{u}_m=\tilde{u}_m(-t)^{\frac{1-m^2}{2}}$ and $\hat{v}_m=\tilde{v}_m(-t)^{\frac{1-m^2}{2}}$.
       Then $\hat{u}, \hat{v}$ satisfies the linear heat equation
       \begin{align*}
         \frac{\partial \hat{u}}{\partial t}=\frac{\partial^2 \hat{u}}{\partial z_1^2}+\frac{\partial^2 \hat{u}}{\partial z_2^2}
       \end{align*}
       and $| \hat{u}_m| +| \hat{v}_m| \leq (C(L)\varepsilon\varepsilon_0+C\varepsilon)(-t)^{1-\frac{m^2}{2}}$ in $\Omega_{L_0/4}\times [-L_0^2/16,-1]    $

       The heat kernel with Dirichlet Boundary for $\Omega_{L_0/4}$ is
       \begin{align*}
          K_t&(x,y)=-\frac{1}{4\pi t}\sum_{\delta_i\in \{\pm 1\}, k_i\in \mathbb{Z}} (-1)^{-(\delta_1+\delta_2)/2}\cdot\\ &\exp\left({-\frac{\left\lvert (x_1,x_2)-(\delta_1 y_1, \delta_2 y_2)-(1-\delta_1,1-\delta_2)\frac{L_0}{4}+(4k_1,4k_2)\frac{L_0}{4}\right\rvert ^2}{4t}}\right)
       \end{align*}
       where $x=(x_1,x_2), y=(y_1, y_2)$ are in $\Omega_{L_0/4}$.

       $K_t$ satisfies
       \begin{enumerate}
         \item $\partial_tK_t=\Delta_x K_t$ in  $\Omega_{L_0/4}\times (0,\infty) $
         \item $K_t$ is symmetric in $x$ and $y$
         \item $\lim_{t\rightarrow 0}K_t(x,y)=\delta(x-y)$ in the distribution sense.
         \item $K_t(x,y)=0$ \ for $y\in\partial\Omega_{L_0/4}$ and $t>0$
       \end{enumerate}
       The solution formula is (we demonstrate it for $\hat{u}_m$, the same for $\hat{v}_m$)
       \begin{align*}
         \hat{u}_m(x,t)= &\int_{\Omega_{L_0/4}}K_{t+L_0^2/16}(x,y)\hat{u}_m(y,-\frac{L_0^2}{16})dy \\ -&\int_{-L_0^2/16}^{t}\int_{\partial\Omega_{L_0/4}}\partial_{\nu_y}K_{t-\tau}(x,y)\hat{u}_m(y,\tau)dy \ d\tau
       \end{align*}

       Now we let $(x,t)\in \Omega_{L_0/100}\times [-L_0^2/100^2,-1]$ ($L_0$ is large enough) and $-L_0^2/16\leq\tau<t$.
       For all such $(x,t)$ and $\tau$ we have the following heat kernel estimate: 

         \begin{align}\label{hkest1}
           \int_{\Omega_{L_0/4}}| K_{t+L_0^2/16}(x,y)| dy\leq C
         \end{align}

        \begin{align}\label{hkest2}
          \int_{\partial{\Omega_{L_0/4}}}| \partial_{\nu_y}K_{t-\tau}(x,y)| dy\leq \frac{CL_0^2}{(t-\tau)^2}e^{-\frac{L_0^2}{1000(t-\tau)}}
        \end{align}

        (\ref{hkest1}) is a standard result that the integration of the Dirichlet heat kernel on an open domain is  no more than 1 at each positive time.

       A brief explanation of (\ref{hkest2}) is in the Appendix \ref{appendix heat kernel}.

       Now we can estimate $\hat{u}_m(\bar{x},\bar{t})$ for $(\bar{x},\bar{t})\in \Omega_{L_0/100}\times [-\frac{L_0^2}{100^2},-1]$.
       \begin{align*}
         | \hat{u}_m(x,t)| &\leq (C(L_0)\varepsilon_0\varepsilon+C\varepsilon)\Big(\frac{L_0^2}{16}\Big)^{1-\frac{m^2}{2}} \\
         &+(C(L_0)\varepsilon_0\varepsilon+C\varepsilon)\int_{-L_0^2/16}^{t}\frac{CL_0^2}{(t-\tau)^2}e^{-\frac{L_0^2}{1000(t-\tau)}}(-\tau)^{1-\frac{m^2}{2}}d\tau\\
       \end{align*}

       For $L_0$ large and $-200^2\leq t\leq -1$
       \begin{align*}
         \frac{CL_0^2}{(t-\tau)^2}e^{-\frac{L_0^2}{1000(t-\tau)}}(-\tau)^{\frac{1}{2}}\leq CL_0^{-1}
       \end{align*}
       whenever $\tau<t$.

       Therefore
       \begin{align}\label{Heat Equation estimate}
         | \hat{u}_m(x,t)| &\leq (C(L_0)\varepsilon_0\varepsilon+C\varepsilon)\Big(\frac{L_0^2}{16}\Big)^{1-\frac{m^2}{2}} \\
         &+(C(L_0)\varepsilon_0\varepsilon+C\varepsilon)\int_{-L_0^2/16}^{t}CL_0^{-1}(-\tau)^{\frac{1-m^2}{2}}d\tau \notag\\
         &\leq(C(L_0)\varepsilon_0\varepsilon+C\varepsilon)\Big[\Big(\frac{L_0^2}{16}\Big)^{1-\frac{m^2}{2}} +L_0^{-1}\frac{4}{m^2-3}(-t)^{\frac{3-m^2}{2}}  \Big] \notag
       \end{align}
       The same estimate holds for $\hat{v}_m$.

       Choose $L_0$ large enough such that $[-200,200]^2\subset\Omega_{L_0/100}$.

       Summing up (\ref{Heat Equation estimate}) over $m\geq 2$, remembering that
       \begin{align*}
         | \tilde{u}_m|  & = | \hat{u}_m|  (-t)^{\frac{m^2-1}{2}} \\
         | \tilde{v}_m|  & = | \hat{v}_m|  (-t)^{\frac{m^2-1}{2}}
       \end{align*}
       we then conclude that for $(x,t)\in\Omega_{200}\times[-200^2,-1]$,
       \begin{align}\label{m >= 2 mode conclusion}
         \sum_{m=2}^{\infty}| \tilde{u}_m| +|
         \tilde{v}_m| \leq & \sum_{m=2}^{\infty}(C(L_0)\varepsilon_0\varepsilon+C\varepsilon)\Big[\Big(\frac{L_0^2}{16(-t)}\Big)^{\frac{2-m^2}{2}} +L_0^{-1}\frac{4}{m^2-3}(-t)^{\frac{1}{2}}\Big]\\
         \leq & C(L_0)\varepsilon_0\varepsilon+CL_0^{-1}\varepsilon \notag
       \end{align}

       \noindent\textbf{Step 5:} In the next we consider the case that $m=1$.
       $\tilde{u}_1,\tilde{v}_1$ satisfies the linear heat equation
       \begin{align*}
         \frac{\partial \tilde{u}}{\partial t}=\frac{\partial^2 \tilde{u}}{\partial z_1^2}+\frac{\partial^2 \tilde{u}}{\partial z_2^2}  \ \  \text{  in } \Omega_{L_0/4}\times [-L_0^2/16,-1]
       \end{align*}
       For each $(X_0,t_0)=(x_1,x_2,t_0)\in\Omega_{L_0/4} \times [-L_0^2/16,-1]$,  taking the Fourier coefficient of $\cos\theta, \sin\theta$ of (\ref{refinedtuneK}), we can find $a_i, b_i, i=0,1,2$ satisfying $| a_i|
       +| b_i| \leq C(L_0)\varepsilon$ and
        \[| \tilde{u}_1-(a_0+a_1z_1+a_2z_2)|
        +| \tilde{v}_1-(b_0+b_1z_1+b_2z_2)| \leq (C(L_0)\varepsilon\varepsilon_0+C\varepsilon)(-t)^{1/2}\]
       in $\Omega_{(-t_0)^{\frac{1}{2}}}(x_1,x_2)\times[0,2\pi]\times[2t_0,t_0]$.

       The linear term satisfies the heat equation trivially, so the parabolic interior gradient estimate implies that $|
       \frac{\partial^2 \tilde{u}_1}{\partial z_i\partial z_j}| +| \frac{\partial^2 \tilde{v}_1}{\partial z_i \partial z_j}| \leq(C(L_0)\varepsilon\varepsilon_0+C\varepsilon)(-t)^{-1/2}$ for each pair of $1\leq i,j\leq2$ in  $\Omega_{L_0/8}\times[-L_0^2/64,-1]$.

       Note that $\frac{\partial^2 \tilde{u}_1}{\partial z_i\partial z_j}$ and $\frac{\partial^2 \tilde{v}_1}{\partial z_i\partial z_j}$ also satisfies the heat equation in $\mathbb{R}^2$,
       we can apply the same argument as in the case of $m\geq 2$  to $\frac{\partial^2 \tilde{u}_1}{\partial z_i\partial z_j}$ and $\frac{\partial^2 \tilde{v}_1}{\partial z_i\partial z_j}$ respectively (i.e. express them as the integral of their boundary data using the heat kernel) to obtain that:
         \begin{align*}
         |
         \frac{\partial^2 \tilde{u}_1}{\partial z_i\partial z_j}| &\leq (C(L_0)\varepsilon_0\varepsilon+C\varepsilon)\Big(\frac{L_0^2}{16}\Big)^{-\frac{1}{2}} \\
         &+(C(L_0)\varepsilon_0\varepsilon+C\varepsilon)\int_{-L_0^2/16}^{t}\frac{CL_0^2}{(t-\tau)^2}e^{-\frac{L_0^2}{1000(t-\tau)}}(-\tau)^{-\frac{1}{2}}d\tau\\
       \end{align*}
       for each $1\leq i, j\leq 2$. Note that the same inequality holds for $\frac{\partial^2 \tilde{v}_1}{\partial z_i\partial z_j}$.

       For $L_0$ large, $-200^2\leq t\leq -1$.
       \begin{align*}
         \frac{CL_0^2}{(t-\tau)^2}e^{-\frac{L_0^2}{1000(t-\tau)}}\leq CL_0^{-2}
       \end{align*}
       whenever $\tau<t$.

       Then we have:

       \begin{align}\label{mode1estimate}
         \sum_{i,j=1,2}| \frac{\partial^2 \tilde{u}_1}{\partial z_i\partial z_j}| +|
         \frac{\partial^2 \tilde{v}_1}{\partial z_i\partial z_j}| \leq C(L_0)\varepsilon_0\varepsilon+CL_0^{-1}\varepsilon
       \end{align}
      for $(x,t)\in\Omega_{200}\times[-200^2,-1]$.
      This means we can find real numbers $A_i, B_i$, $i=0,1,2$ such that
      \begin{align}\label{mode1conclusion}
        | \tilde{u_1}-(A_0+A_1z_1+A_2z_2)| +|
        \tilde{v_1}-(B_0+B_1z_1+B_2z_2)| \leq C(L_0)\varepsilon_0\varepsilon+CL_0^{-1}\varepsilon
      \end{align}
      and $| A_i| +|
      B_i| \leq C(L_0)\varepsilon$. \\

      \noindent\textbf{Step 6:} Finally we deal with the case that $m=0$. Using the polar coordinate on $M_t$, the normal vector $\nu$ satisfies:
      \[\nu\sqrt{1+r^{-2}(\frac{\partial r}{\partial \theta})^2+(\frac{\partial u}{\partial z_1})^2+(\frac{\partial u}{\partial z_2})^2}=(\cos\theta,\sin\theta,-\frac{\partial r}{\partial z_1},-\frac{\partial r}{\partial z_2})+(\frac{r_\theta}{r}\sin\theta, -\frac{r_\theta}{r}\cos\theta, 0, 0)\]
      Notice that $K=(-r\sin\theta, r\cos\theta, 0 , 0)$, thus
      \begin{align*}
        \int_{0}^{2\pi}u\sqrt{1+r^{-2}(\frac{\partial r}{\partial \theta})^2+(\frac{\partial u}{\partial z_1})^2+(\frac{\partial u}{\partial z_2})^2}d\theta=\int_{0}^{2\pi}-\frac{dr}{d\theta}d\theta=0
      \end{align*}
      Since $| u|
      \leq C(L_0)\varepsilon$ and $| \frac{\partial r}{\partial\theta}| +| \frac{\partial u}{\partial z_1}|
      +| \frac{\partial u}{\partial z_2}| \leq C\varepsilon_0$, moreover $r>\frac{1}{\sqrt{2}}-C(L_0)\varepsilon_0>\frac{1}{2}$, we conclude that
      \begin{align}\label{mode0estimate}
        \left\lvert \int_{0}^{2\pi}u(z_1,z_2,\theta)d\theta\right\rvert \leq C(L_0)\varepsilon_0\varepsilon
      \end{align}
      in $\Omega_{200}\times[-200^2,-1]$.

      Considering (\ref{diff tildeu and u}), we have $| \tilde{u}_0|
      \leq C(L_0)\varepsilon_0\varepsilon$ in $\Omega_{200}\times[-200^2,-1]$. Finally note that $\tilde{v}_0=0$ trivially.\\

      \noindent\textbf{Step 7: }Combing the analysis for all the $\tilde{u}_i$
      (\ref{m >= 2 mode conclusion}) (\ref{mode1conclusion}) (\ref{mode0estimate}) (and the same hold for $\tilde{v}_i$), we conclude that there exists real numbers $A_i, B_i$, $i=0,1,2$ such that
      \begin{align*}
        | \tilde{u}-(A_0+A_1z_1+A_2z_2)\cos\theta| +| \tilde{v}-(B_0+B_1z_1+B_2z_2)\sin\theta| \leq C(L_0)\varepsilon_0\varepsilon+CL_0^{-1}\varepsilon
      \end{align*}
      in $\Omega_{200}\times[-200^2,-1]$
      and $| A_i| +| B_i| \leq C(L_0)\varepsilon$.

      Considering the fact that $| u-\tilde{u}| \leq C(L_0)\varepsilon_0\varepsilon$, we have
      \begin{align*}
        | u-(A_0+A_1z_1+A_2z_2)\cos\theta| +| v-(B_0+B_1z_1+B_2z_2)\sin\theta| \leq C(L_0)\varepsilon_0\varepsilon+CL_0^{-1}\varepsilon
      \end{align*}
      in $\Omega_{200}\times[0,2\pi]\times[-200^2,-1]$.

      So there exists another normalized rotation vector field $\tilde{K}$ such that $| \left<\tilde{K},\nu\right>| \leq C(L_0)\varepsilon_0\varepsilon+CL_0^{-1}\varepsilon$
      in $\hat{\mathcal{P}}(\bar{x},-1,100,100^2)$.
      Also $H\leq 2$ in this parabolic neighbourhood.
      Note that in the step 1 we actually get $| \bar{K}| H\leq 1+C(L_0)\varepsilon_0$, hence $| \tilde{K}| H\leq 1+C(L_0)(\varepsilon+\varepsilon_0)$.

      Now we first choose $L_0$ large enough such that $CL_0^{-1}<1/4$ and then choose $\varepsilon_0$ small enough depending on $L_0$ such that $C(L_0)\varepsilon_0<1/4$.
      Then $(\bar{x},-1)$ is $\varepsilon/2$ symmetric.

\end{proof}

Unlike in  \cite{brendle2019uniqueness, brendle2018uniqueness}, there is another type of canonical neighborhood. By the result of Brendle--Choi \cite{brendle2019uniqueness}, it turns out that there is only one more that has never been handled before, namely the Bowl$\times\mathbb{R}$. The details will be explained in the next section. This type of neighborhood will be handled by
 Theorem \ref{bowl x R improvement}. Some aspects of the proof are inspired by the Theorem 5.4 of \cite{brendle2019uniqueness} and the Proposition 2.6 of \cite{angenent2020uniqueness}.

 The constant of the Theorem \ref{bowl x R improvement} needs to depend on the constant of the Theorem \ref{Cylindrical improvement} because we need to repeatedly apply the Theorem \ref{Cylindrical improvement} in the proof.

\begin{Th}\label{bowl x R improvement}
There exists a constant  $L_1\gg L_0$ and  $0<\varepsilon_1\ll \varepsilon_0$ such that: for a mean curvature flow solution $M_t$ and a space-time point $(\bar{x},\bar{t})$,
if $\hat{\mathcal{P}}(\bar{x},\bar{t},L_1,L_1^2)$ is $\varepsilon_1$ close to a piece of $\text{Bowl}^{2} \times \mathbb{R}$ in $C^{10}$ norm after the parabolic rescaling (which makes $H(\bar{x},\bar{t})=1$)
and every points in $\hat{\mathcal{P}}(\bar{x},\bar{t},L_1,L_1^2)$ are $\varepsilon$ symmetric with $0<\varepsilon\leq \varepsilon_1$, then $(\bar{x},\bar{t})$ is $\frac{\varepsilon}{2}$ symmetric.
\end{Th}

\begin{proof}
  Throughout the proof, $L_1$ is always assumed to be large enough depending only on $L_0,\varepsilon_0$ and $\varepsilon_1$ is assumed to be small enough depending  on $L_1, L_0, \varepsilon_0$. Also the constant $C$ depend only on
  $L_0, \varepsilon_0$ from Lemma \ref{Cylindrical improvement}.

  As in Lemma \ref{NRVFDistanceBowl}, we denote by $\Sigma$ the standard Bowl soliton in $\mathbb{R}^3$. That means the tip of $\Sigma$ is the origin, the mean curvature at the tip is $1$ and the rotation axis  is $x_3$, also it enclose the positive part of $x_3$ axis.
  Let $\omega_3$ be the unit vector in the $x_3$ axis that points to the positive part of $x_3$ axis (so $\omega_3$ coincide with the inward normal vector of $\Sigma$ at the origin).
  We write $\Sigma_t=\Sigma+\omega_3 (t+1)$, then $\Sigma_t$ is the translating mean curvature flow solution and $\Sigma=\Sigma_{-1}$.

  We can do a parabolic rescaling and  space-time translation such that $H(\bar{x},\bar{t})=1$ and $\bar{t}=-1$ .
  Moreover, there exists a scaling factor $\kappa>0$ such that the parabolic neighborhood $\hat{\mathcal{P}}(\bar{x},\bar{t},L_1,L_1^2)$ can be approximated by the family of translating $\kappa^{-1}\Sigma_t\times \mathbb{R}$ with an error that is $\varepsilon_1$ small in $C^{10}$ norm.
  i.e. can be written as a graph over the corresponding parabolic neighborhood in space-time track of $\kappa^{-1}\Sigma_t\times \mathbb{R}$ with graph norm $\varepsilon_1$ small in $C^{10}$ (See Definition \ref{closenessdef} and Remark \ref{closenessdef rmk}).

  By the above setup, the maximal mean curvature of $\kappa^{-1}\Sigma_t$ is $\kappa$. Let $p_t=\omega_3(t+1)\in\kappa^{-1}\Sigma_t$ be the tip.
  Let the straight line $l_t=\{p_t\}\times\mathbb{R}\subset \kappa^{-1}\Sigma_t\times
  \mathbb{R}$. The splitting direction $\mathbb{R}$ is assumed to be $x_4$ axis.

  By our normalization $H(\bar{x},-1)=1$, the maximality of $\kappa$ together with the approximation ensures that $\kappa\geq 1-C\varepsilon_1$.

   Define $d(\bar{x},l_t)=\min_{a\in l_t}| \bar{x}-a| $ to be the Euclidean distance between $\bar{x}$ and the line $l_t$.

   We divide the proof into several steps, the first step deals with the case that $\bar{x}$ is far away from $l_{-1}$ in which we can apply cylindrical improvement. The remaining steps deal with the case that $\bar{x}$ is not too far from $l_{-1}$. \\

   \noindent\textbf{Step 1:} We have the following  scale invariant statement  by the structure of Bowl soliton:

   There exists $\Lambda_{\star}\gg 1$ depending on $L_0,\varepsilon_0$ (not on $\kappa$) such that,
   if $(p,t)\in \kappa^{-1}\Sigma_t\times\mathbb{R}$ with $H(p,t)d(p,l_t)>\Lambda_{\star}/2$,
    then every point in $\hat{\mathcal{P}}_1(p,t,2L_0,4L_0^2)$ is $(\varepsilon_0/2,200)$ cylindrical, where $\hat{\mathcal{P}}_1$  denotes the parabolic neighborhood of $\kappa^{-1}\Sigma_t\times \mathbb{R}$.
  Moreover,  for each point $(q,s)\in \hat{\mathcal{P}}_1(p,t,2L_0,4L_0^2)$, the parabolic neighborhood
  $\hat{\mathcal{P}}_1(q,s,200^2,200^2)$ is contained in $\hat{\mathcal{P}}_1(p,t,5\cdot 10^5L_0,25\cdot10^{10}L_0^2)$, and thus is contained in $\hat{\mathcal{P}}_1(p,t,L_1/2,L_1^2/4)$ provided that $L_1$ is large enough compared to $L_0$.

  Since $\hat{P}(\bar{x},-1,L_1, L_1^2)$ can be written as a graph over the corresponding parabolic neighborhood in space-time track of $\kappa^{-1}\Sigma_t\times \mathbb{R}$ with $\varepsilon_1$ small graph norm (in $C^{10}$),
  the above statement is true within $\hat{P}(\bar{x},-1,L_1, L_1^2)$ with worse constants.
  In fact, we can simply use the graph map (See Remark \ref{closenessdef rmk}) to give an diffeomorphism between two hypersurfaces which is $\varepsilon_1$ close to the identity map (See definition \ref{closenessdef}).
  Under this map,  the size of the neighborhoods differ by at most $C(L_0, L_1)\varepsilon_1$. Also, being $(\varepsilon_0/2,200)$ cylindrical deteriorates to being $(\varepsilon_0/2 + C\varepsilon_1,200-C\varepsilon_1)$ cylindrical.
  Since we can choose $\varepsilon_1$ sufficiently small after $L_0,L_1,\varepsilon_0$, all above factors $C\varepsilon_1$ can be made much smaller than $\varepsilon_0/2$.
  By such an approximation, we get:

   If $d(\bar{x},l_{-1})>\Lambda_{\star}$, then for every point $(y,s)\in\hat{\mathcal{P}}(\bar{x},-1,L_0,L_0^2)$ the parabolic neighborhood $\hat{\mathcal{P}}(y,s,100^2,100^2)\subset\hat{\mathcal{P}}(\bar{x},-1,L_1,L_1^2)$ and $(y,s)$ is $(\varepsilon_0,100)$ cylindrical.

    Therefore in the case of $d(\bar{x},l_{-1})\geq \Lambda_{\star}$, the conditions of cylindrical improvement are satisfied. Hence $(\bar{x},-1)$ is $\frac{\varepsilon}{2}$ symmetric and we are done.\\

   In the following steps we assume that $d(\bar{x},l_{-1})<\Lambda_{\star}=\Lambda_{\star}(L_0,\varepsilon_0)$\\

  \noindent\textbf{Step 2:}
  By the asymptotic behaviour of $\Sigma_t$, we have the scale invariant identity:
  \[\frac{\kappa}{H(p)}=f(H(p)d(p,l_t))\] where $p\in \kappa^{-1}\Sigma_{t}\times\mathbb{R}$ and $f$ is a continuous increasing function such that $f(0)=1$.
  In fact $f(x)=O(x)$ as $x\rightarrow+\infty$.

  Since $\hat{P}(\bar{x},-1,L_1, L_1^2)$ is a graph over the corresponding parabolic neighborhood (in space-time track) of $\kappa^{-1}\Sigma_t\times \mathbb{R}$ with $\varepsilon_1$ small graph norm (in $C^{10}$),
  it's possible to find a point $p\in \kappa^{-1}\Sigma_{-1}\times\mathbb{R}$ with $| p-\bar{x}| \leq \varepsilon_1$ and $| H(p)-1| \leq C\varepsilon_1$,
  therefore $d(p,l_{-1})\leq\Lambda_{\star}+C\varepsilon_1$, thus we have the upper and lower bound for $\kappa$:
  \begin{align}\label{upperlowerboundkappa}
     \frac{1}{2}<1-C\varepsilon_1<\kappa<f((1+C\varepsilon_1)(\Lambda_{\star}+\varepsilon_1))<C\Lambda_{\star}
  \end{align}
  This means $\kappa^{-1}\Sigma_{t}\times\mathbb{R}$ is equivalent to the standard Bowl$\times\mathbb{R}$ up to a scaling factor depending on $L_0,\varepsilon_0$. \\

  \noindent\textbf{Step 3:}
   We define a series of set $\text{Int}_j(\hat{\mathcal{P}}(\bar{x},-1,L_1,L_1^2))$ inductively:
   \begin{align*}
   &\text{Int}_{0}(\hat{\mathcal{P}}(\bar{x},-1,L_1,L_1^2))=\hat{\mathcal{P}}(\bar{x},-1,L_1,L_1^2)\\
  &(y,s)\in \text{Int}_j(\hat{\mathcal{P}}(\bar{x},-1,L_1,L_1^2))\Leftrightarrow
  \hat{\mathcal{P}}(y,s,10^6L_0,10^{12}L_0^2)\subset \text{Int}_{j-1}(\hat{\mathcal{P}}(\bar{x},-1,L_1,L_1^2))
  \end{align*}
  We  abbreviate $\text{Int}_j(\hat{\mathcal{P}}(\bar{x},-1,L_1,L_1^2))$ as $\text{Int}(j)$. It's clear that Int($j$) is a decreasing set that are all contained in $\hat{\mathcal{P}}(\bar{x},-1,L_1,L_1^2)$. \\

    We claim that there exists $\Lambda_1\gg \Lambda_{\star}$, depending only on $\varepsilon_0, L_0$, such that for any point $(x,t)\in \text{Int}_j(\hat{\mathcal{P}}(\bar{x},-1,L_1,L_1^2))$, if $d(x,l_t)\geq 2^{\frac{j}{100}}\Lambda_1$, then $(x,t)$ is $2^{-j}\varepsilon$ symmetric. $j=0,1,2,...$\\

  To achieve this, we do induction on $j$. By the assumption in the theorem, $j=0$ is automatically true. Now suppose that the statement is true for $j-1$.

  Similar to step 1, by the structure of the Bowl soliton we can find a constant $\Lambda_1$ large enough depending only on $L_0,\varepsilon_0$ with the following property: if the point $q\in \kappa^{-1}\Sigma_{t}\times\mathbb{R}$ satisfies $d(q,l_t)\kappa\geq \Lambda_{1}/2$,
  then every point $(q',t')\in\hat{\mathcal{P}}_1(q,t,L_0,L_0^2)$ is $(\varepsilon_0/2,200)$ cylindrical and
  $\hat{\mathcal{P}}_1(q',t',200^2,200^2)\subset\hat{\mathcal{P}}_1(q,t,5\cdot10^5L_0,25\cdot10^{10}L_0^2)$.
  Moreover, $d(q,l_t)H(q)>2000L_0$ holds.
  Reasoning as in step 1, the $\varepsilon_1$ approximation gives us the same statement within  $\hat{P}(\bar{x},-1,L_1, L_1^2)$ with worse constants:

   \begin{align}\label{bigdistimpliescylindrical}
     &\text{If }  d(x,l_t)\geq \Lambda_{1},
    \text{then every point }(y,s)\in\hat{\mathcal{P}}(x,t,L_0,L_0^2) \text{ is } (\varepsilon_0,100) \text{ cylindrical}\text{ and }  \\ & d(x,l_t)H(x)\geq 1000L_0. \text{ Moreover } \hat{\mathcal{P}}(y,s,100^2,100^2)\subset\mathcal{P}(x,t,10^6L_0,10^{12}L_0^2)\notag
   \end{align}

   \begin{rem}
     Validation of  $(\varepsilon_0,100)$ cylindrical condition involves a rescaling process, while the mean curvature at the point $(q,t)$ in our context may be much smaller than 1. Indeed it depends on $L_1$, thus the scaling factor also depends on $L_1$.
   But this is not an issue because  it will be conquered by $\varepsilon_1$, since $\varepsilon_1$ is chosen after $L_1$.
   \end{rem}

    Recall the approximation by graph map described in step 1. Since
    the graph norm is $\varepsilon_1$ small in $C^{10}$, the distance quantities and normal vectors only differ by $C\varepsilon_1$. Under such an approximation and structure of the Bowl soliton, we know that if $(x,t)\in\hat{\mathcal{P}}(\bar{x},0,L_1,L_1^2)$ with $d(x,l_t)\geq 1 \geq \frac{1}{2}\kappa^{-1}$, then
    \begin{itemize}
      \item $\frac{\left<\nu, x-x'\right>}{| x-x'| }\leq C\varepsilon_1$
      \item $\frac{\left<\omega_n, x-x'\right>}{| x-x'| }\geq C^{-1}$
    \end{itemize}
    where $x'$ is a point on $l_t$ such that $| x-x'| =d(x,l_t)$.

    Let $x(t)$ be the trajectory of $x$ under mean curvature flow, such that
    $x(t_0)=x_0$ and $d(x_0,l_{t_0})> 1$, then for $t\leq t_0$
    \begin{align}\label{distancetotipdecreasing}
      \frac{d}{dt}d(x(t),l_t)^2&= \left<x-x',H(x)\nu(x)-\kappa\omega_n\right>\\
      &\leq| x-x'| (C\varepsilon_1 - C^{-1}) < 0 \notag
    \end{align}
    Therefore $x(t)$ increases as $t$ decreases, as long as $x(t)\in\hat{\mathcal{P}}(\bar{x},0,L_1,L_1^2)$.

    Now suppose that $(x,t)\in\text{Int}(j)$ and $d(x,l_t)\geq 2^{\frac{j}{100}}\Lambda_1.$
    For any point $(\tilde{x},t)\in\hat{\mathcal{P}}(x,t,L_0,L_0^2)$, by (\ref{bigdistimpliescylindrical})
    \begin{align}\label{distance increasing1}
      d(\tilde{x},l_t)&\geq d(x,l_t)-d(x,\tilde{x})\\
      &\geq d(x,l_t)-L_0H(x)^{-1}\notag\\
      &\geq (1-1/1000)d(x,l_t)\notag\\
      &\geq (1-1/1000) 2^{\frac{j}{100}}\Lambda_1>2^{\frac{j-1}{100}}\Lambda_1 \notag
    \end{align}
    Additionally, every point $(y,s)\in\hat{\mathcal{P}}(x,t,L_0,L_0^2)$ must be in Int($j-1$), hence by (\ref{distancetotipdecreasing}) and (\ref{distance increasing1}) $d(y,l_s)>2^{\frac{j-1}{100}}\Lambda_1$ and therefore is $2^{-j+1}\varepsilon$ symmetric by induction hypothesis.

    Together with (\ref{bigdistimpliescylindrical}), the condition of the cylindrical improvement are satisfied. Hence $(x,t)$ is $2^{-j}\varepsilon$ symmetric. So the conclusion is true for $j$.\\

   \noindent\textbf{Step 4:} Recall that $d(\bar{x},l_{-1})\leq \Lambda_{\star}$. We first show the following:

    \begin{align}\label{sizeofInt}
      &\hat{\mathcal{P}}(\bar{x},-1,10^{-8}L_0^{-1}\hat{L},10^{-16}L_0^{-2}\hat{L}^2)
    \text{ is contained in } \text{Int}_1\hat{\mathcal{P}}(\bar{x},-1,\hat{L},\hat{L}^2)\\\notag
    &\text{ whenever  } 10^{-16}L_0^{-2}\hat{L}>\Lambda_1 \text{ and } \hat{L}<L_1
    \end{align}.

   To see this, we notice that the statement with stronger constants hold in $\kappa^{-1}\Sigma_t\times\mathbb{R}$, namely:
    if $p\in \kappa^{-1}\Sigma_{-1}\times\mathbb{R}$ satisfies $d(p,l_{-1})<2\Lambda_{\star}$, then
    $\hat{\mathcal{P}}_1(p,-1,10^{-7}L_0^{-1}\hat{L},10^{-14}L_0^{-2}\hat{L}^2)
    $ is contained in $ \text{Int}_1\hat{\mathcal{P}}_1(p,-1,\hat{L},\hat{L}^2)$ whenever $10^{-7}L_0^{-1}\hat{L}>\Lambda_1$. This can be verified by the asymptotic behaviour of the Bowl soliton.
    Reasoning as in step 1, the $\varepsilon_1$ approximation gives us the same statement within  $\hat{P}(\bar{x},-1,L_1, L_1^2)$ with worse constants, (\ref{sizeofInt}) is true.

    By (\ref{sizeofInt}) and induction on $j$ we obtain:
    \begin{align}\label{size of Int 2}
      \hat{\mathcal{P}}(\bar{x},-1,10^{-8j}L_0^{-j}L_1,10^{-16j}L_0^{-2j}L_1)\subset\text{Int}(j) \text{ \ if } 10^{-8j-8}L_0^{-j-1}L_1>\Lambda_1
    \end{align}

    By (\ref{size of Int 2}) and Step 3 we immediately get that:\\

    For any point   $(x,t)\in\hat{\mathcal{P}}(\bar{x},-1,10^{-8j}L_0^{-1}L_1,10^{-16j}L_0^{-2}L_1)$ with
    $d(x,l_t)\geq 2^{\frac{j}{100}}\Lambda_1$ and $j<\frac{1}{8}\log_{10}(\frac{L_1}{L_0\Lambda_1})$
    is $2^{-j}$ symmetric. \\

    Next we show that there exists a constant $C$ independent of $L_1$
    such that $(x,t)$ is $C \cdot 2^{-j/4}\varepsilon$ symmetric for any   $j<\frac{1}{8}\log_{10}(\frac{L_1}{L_0\Lambda_1})$. Subsequently we can take
    a large $L_1$ such that $j$ is allowed to be $[4\log_2(2C)]+4$ and we are done.\\

    \noindent\textbf{Step 5:} Define the region

    \begin{align*}
      \Omega_j^t=&\{(y,t)\big\lvert  | y_4| \leq W_j, d(y,l_t)\kappa\leq D_j\}\\
      \Omega_j=&\bigcup\limits_{t\in[-1-T_j,-1]} \Omega_j^t\\
      \partial^1\Omega_j=&\{(y,s)\in\partial\Omega_j| d(y,l_s)\kappa= D_j\}\\
      \partial^2\Omega_j=&\{(y,s)\in\partial\Omega_j\big\lvert  | y_n| =W_j| \}
    \end{align*}
    where $D_j=2^{\frac{j}{100}}\Lambda_1$, $W_j=T_j=2^{\frac{j}{50}}\Lambda_1^2$.

     Now every point in $\partial^1\Omega_j^t$ is $2^{-j}\varepsilon$ symmetric by the previous argument. But the symmetry at each point is realized by different normalized rotation vector fields,
     we claim next that there is a single vector field that does the job.\\

    \noindent\textbf{Claim:} There exists a normalized rotation vector field $K_j$ such that
    \begin{align}\label{}
     & | \left<K_j,\nu\right>| H\leq C_1(W_j+D_j+T_j)^2 2^{-j}\varepsilon \text{ on }
     \partial^1\Omega_j^t \\
     & | \left<K_j,\nu\right>| H \leq C_1(W_j+D_j+T_j)^2 \varepsilon  \text{ on } \Omega_j \
     (\text{in particular, on }   \partial^2\Omega_j\cup \Omega_j^{-1-T_j}
     )
    \end{align}
     where $C_1=C_1(\kappa)=C_1(L_0)$ \\

    \textit{Proof:} First of all, for any point $(y,s)\in\text{Int}(j)$  we can find a normalized rotation field, denoted by $K^{(y,s)}$ such that
    $| \left<K^{(y,s)},\nu\right>| H\leq \varepsilon$ and $| K| H<5$ in the parabolic ball
    $\hat{\mathcal{P}}(y,s,100,100^2)$. Moreover, if $2^{\frac{j}{100}}\Lambda_1\leq d(y,l_s)<2^{\frac{j+1}{100}}\Lambda_1$, then
    $|
    \left<K^{(y,s)},\nu\right>| H\leq 2^{-j}\varepsilon$.

    Let $K_0(x)=Jx$. It's easy to check that $\left<K_0,\nu\right>=0$ and $| K_0| H<2$  on $\kappa^{-1}\Sigma_t\times\mathbb{R}$. (See Appendix \ref{ode of bowl} for details). By the approximation we know that $| \left< K_0,\nu\right>| H<2\varepsilon_0<\varepsilon_{bowl}$ and $| K_0| H<2+2\varepsilon_0<5$ on $M_t$.

    For any point $(z,\tau)\in\Omega_j$, a large parabolic neighborhood $\hat{\mathcal{P}}(z,\tau,100,100^2)$ is $\varepsilon_0<\varepsilon_{cyl}/4$ close to a piece of $\kappa^{-1}\Sigma_t\times\mathbb{R}$ after the renormalization.
    Together with the structure of the Bowl soliton the ball $B_g(z,H(z,\tau)^{-1})$ is either $\varepsilon_{cyl}$ close to $S^1\times\mathbb{R}$ (after renormalization) or has  $\lambda_1+\lambda_2\geq \varepsilon_{cyl}H/4 $.

     Therefore either Lemma \ref{NRVFDistanceCyl} or Lemma \ref{NRVFDistanceBowl} applies and we get:
     \[\min\Big\{\max\limits_{B_{LH(z,\tau)^{-1}}(z)}| K_0-K^{(z,\tau)}| H,\max\limits_{B_{LH(z,\tau)^{-1}}(z)}| K_0+K^{(z,\tau)}| H\Big\}<C(L+1)\varepsilon_1\]
     If we take $L=10\kappa^2L_1^4$, then $B_{LH(z,\tau)^{-1}}(z)$  contains every time slice of     $\hat{\mathcal{P}}(\bar{x},-1,L_1,L_1^2)$. Therefore we have :
     \[\min\Big\{\max\limits_{\hat{\mathcal{P}}(\bar{x},-1,L_1,L_1^2)}| K_0-K^{(z,\tau)}| H,\max\limits_{\hat{\mathcal{P}}(\bar{x},-1,L_1,L_1^2)}| K_0+K^{(z,\tau)}| H\Big\}<C\kappa^2L_1^4\varepsilon_1<\varepsilon_0\]
    The last inequality holds because we can choose $\varepsilon_1$ after $L_1$.
    In particular:
    \begin{align}\label{KHmax}
      | K^{(z,\tau)}| H<2+3\varepsilon_0<5 \text{ in } \hat{\mathcal{P}}(\bar{x},-1,L_1,L_1^2) \text{ for each } (z,\tau)\in \text{Int}(j)
    \end{align}

    Next, we fix an arbitrary point $(x,-1)\in\partial^1\Omega_j$ and let $\bar{K}=K^{(x,-1)}$.
    Given a point $(y,s)\in \partial^1\Omega_j\cup \partial^2\Omega_j$, we want to show that $K^{(y,s)}$ deviate from $\bar{K}$ by at most $C\varepsilon$.

    If $(y,s)\in \partial^2\Omega_j\subset\Omega_j$, then we can connect $(x,-1)$ and $(y,s)$ by at most $C\kappa(D_j+W_j+T_j)$ consecutive points
     such that the each latter point lies in the parabolic neighborhood of size $\frac{1}{4}$ of the previous point. To be precise,
     we can find $(x_k,t_k)$, $k=0,1,...,N$ such that
    \begin{enumerate}
      \item $ N \leq C\kappa(D_j+W_j+T_j)$
      \item $(x_0,t_0)  =(x,t)$ and $(x_N,t_N)=(y,s)$
      \item  $(x_{k+1},t_{k+1})\in \hat{\mathcal{P}}(x_k,t_k,a,a^2)$
      \item  $(x_k,t_k)\in \Omega_j$
    \end{enumerate}
    where $a=\frac{1}{4}$.
    To see this, we notice that such a sequence of points can be found on the corresponding set of $\kappa^{-1}\Sigma_t\times\mathbb{R}$ with $a=\frac{1}{8}$, then use the approximation described in step 1 to get the corresponding sequence of points on $\Omega_j$.

     The property (3) above ensures that

     $B_{g}(x_{k+1},H(x_{k+1},t_{k+1})^{-1})\times \{t_{k+1}\}\subset\hat{\mathcal{P}}(x_k,t_k,100,100^2)$.

     Again we just need to verify the parallel statement in $\kappa^{-1}\Sigma_t\times\mathbb{R}$ with worse constants and then use approximation, namely:
     if $(p_{k+1},t_{k+1})\in\hat{\mathcal{P}}_1(p_k,t_k,1/2,1/4)$, then

     $B_{g}(p_{k+1},H(p_{k+1})^{-1})\times\{t_{k+1}\}\subset\hat{\mathcal{P}}_{1}(p_k,t_k,50,50^2)$

     By  symmetry assumption and definition \ref{varepsilonsymmetrydef} we know that $| \left<K^{(x_k,t_k)},\nu\right>| H$ and $| \left<K^{(x_{k+1},t_{k+1})},\nu\right>| H$ are both less than $\varepsilon$ on $B_{g}(x_{k+1}, H(x_{k+1},t_{k+1})^{-1})$.

     Together with (\ref{KHmax}) we see that the conditions of either Lemma \ref{NRVFDistanceCyl} or \ref{NRVFDistanceBowl} are satisfied, therefore we may conclude that
     \[\min\left\{
     \begin{array}{c}
       \max\limits_{B_{LH(x_{k+1},t_{k+1})^{-1}}(x_{k+1})}| K^{(x_k,t_k)}-K^{(x_{k+1},t_{k+1})}| H \\
       \max\limits_{B_{LH(x_{k+1},t_{k+1})^{-1}}(x_{k+1})}|
       K^{(x_k,t_k)}+K^{(x_{k+1},t_{k+1})}| H
     \end{array}
     \right\}<C(L+1)\varepsilon\]

     Taking $L=C\kappa(D_j+W_j+T_j)$ with a large universal constant $C$, then $B_{LH(x_{k+1},t_{k+1})^{-1}}(x_{k+1})$ contains every time slices of $\hat{\mathcal{P}}(y,s,100,100^2)$, so we can rewrite the above as:
     \begin{align}\label{Kdifference1}
     \min\left\{
     \begin{array}{c}
       \max\limits_{\hat{\mathcal{P}}(y,s,100,100^2)}| K^{(x_k,t_k)}-K^{(x_{k+1},t_{k+1})}| H \\
       \max\limits_{\hat{\mathcal{P}}(y,s,100,100^2)}| K^{(x_k,t_k)}+K^{(x_{k+1},t_{k+1}}| H
     \end{array}
       \right\}<C\kappa(D_j+W_j+T_j)\varepsilon
     \end{align}

     We sum up (\ref{Kdifference1}) for $k=0,1,...N-1$ and get:
     \begin{align*}
     \min\left\{
     \begin{array}{c}
       \max\limits_{\hat{\mathcal{P}}(y,s,100,100^2)}| \bar{K}-K^{(y,s)}| H\\ \max\limits_{\hat{\mathcal{P}}(y,s,100,100^2)}| \bar{K}+K^{(y,s}| H
     \end{array}
     \right\}
       <CN\kappa(D_j+W_j+T_j)\varepsilon\\
       <C\kappa(D_j+W_j+T_j)^2\varepsilon
     \end{align*}
     Finally we conclude that $| \left<\bar{K},\nu\right>| H<C\kappa(D_j+W_j+T_j)^2\varepsilon$ in $\hat{\mathcal{P}}(y,s,100,100^2)$. Hence $\bar{K}$ is qualified as a normalized rotation vector field that realize the $C_1(D_j+W_j+T_j)^2\varepsilon$ symmetry at any point in $\Omega_j$.
     (and hence any point on $\partial^2\Omega_j$ ), where $C_1=C_1(\kappa)=C_1(L_0)$.

     Next, we consider the case that $(y,s)\in\partial^1\Omega$. The above procedure can still be applied with an extra requirement that  $(x_k,t_k)\in\partial^1\Omega_j$ for each $k$. In fact the upper bound for $N$ can be improved,
     but there will be no essential differences.

     In this case
     $| \left<K^{(x_k,t_k)},\nu\right>| H$ and $| \left<K^{(x_{k+1},t_{k+1})},\nu\right>| H$ are both less than $2^{-j}\varepsilon$ on $B_{H(x_{k+1},t_{k+1})^{-1}}(x_{k+1})$, therefore we have a stronger conclusion:
      \begin{align*}
     \min\left\{
     \begin{array}{c}
       \max\limits_{\hat{\mathcal{P}}(y,s,100,100^2)}| \bar{K}-K^{(y,s)}| H\\ \max\limits_{\hat{\mathcal{P}}(y,s,100,100^2)}| \bar{K}+K^{(y,s}| H
     \end{array}
     \right\}
       <CN\kappa(D_j+W_j+T_j)2^{-j}\varepsilon\\
       <C\kappa(D_j+W_j+T_j)^22^{-j}\varepsilon
     \end{align*}
     claim  is proved.

     With all the claims we are able to do analysis on $\Omega_j$. We take $\bar{K}$ as in the claim  and let $u=\left<\bar{K},\nu\right>$. \\

     \noindent\textbf{Step 6:}

     Let $\Phi(x)=\phi(x_4)=\phi(\left<x,\omega_4\right>)$ where $\phi$ is a one variable function to be specified later in (\ref{definition of phi}).

     Define a function
      \[f(x,t)=e^{-\Phi(x)+\lambda(t-\bar{t})}\frac{u}{H-\mu}\]
       on $M_t$
     where $\lambda,\mu$ will be determined later.

     Use the evolution equation
     \begin{align*}
       \partial_t u & =\Delta u+| A| ^2u\\
       \partial_t H & =\Delta H+ | A| ^2H
     \end{align*}

     We get the evolution equation for $\frac{u}{H-\mu}$:
     \begin{align*}
       (\partial_t -\Delta )\left(\frac{u}{H-\mu}\right)=&\frac{(\partial_t -\Delta )u}{H-\mu}-
       \frac{u(\partial_t -\Delta )H}{(H-\mu)^2}+\frac{2\left<\nabla u,\nabla H\right>}{(H-\mu)^2}-\frac{2| \nabla H| ^2u}{(H-\mu)^3}\\
       =&-\frac{\mu u| A| ^2}{(H-\mu)^2}+2\left<\frac{\nabla H}{H-\mu},\nabla\left(\frac{u}{H-\mu}\right)\right>
     \end{align*}

     Then the evolution equation for $f$:
     \begin{align}\label{EQNforf}
       (\partial_t -\Delta )f=&e^{-\Phi+\lambda(t-\bar{t})}(\partial_t -\Delta )\left(\frac{u}{H-\mu}\right)+(\lambda-\partial_t\Phi+\Delta \Phi-| \nabla \Phi| ^2)f  \\
       &+2e^{-\Phi+\lambda(t-\bar{t})}\left<\nabla \Phi,\nabla\left(\frac{u}{H-\mu}\right)\right>\notag\\
       =&\left(\lambda-\frac{\mu| A| ^2}{H-\mu}-\partial_t\Phi+\Delta \Phi-| \nabla\Phi| ^2\right)f\notag\\
       &+2e^{-\Phi+\lambda(t-\bar{t})}\left<\nabla\left(\frac{u}{H-\mu}\right),\nabla \Phi+\frac{\nabla H}{H-\mu}\right>\notag\\
       =&\left(\lambda-\frac{\mu| A| ^2}{H-\mu}-\partial_t\Phi+\Delta \Phi-| \nabla\Phi| ^2\right)f+2\left<\nabla f,\nabla\Phi+\frac{\nabla H}{H-\mu}\right>\notag\\
       &+2\left<\nabla \Phi,\nabla\Phi+\frac{\nabla H}{H-\mu}\right>f\notag
     \end{align}

     \begin{align*}
         =&\left(\lambda-\frac{\mu| A| ^2}{H-\mu}-\partial_t\Phi+\Delta \Phi +| \nabla\Phi|
       ^2+2\frac{\left<\nabla \Phi,\nabla H\right>}{H-\mu}\right)f \notag\\
       &+2\left<\nabla f,\nabla\Phi+\frac{\nabla H}{H-\mu}\right> \notag
     \end{align*}
 
     We have the following computations:
     \begin{align*}
       &\partial_t\Phi=\phi'(x_4)\left<\partial_t x, \omega_4\right>=\phi'(x_4)\left<\vec{H},\omega_4\right>\\
       &\Delta\Phi=\phi'(x_4)\Delta x_4+\phi''(x_4)| \nabla x_4| ^2=\phi'(x_4)\left<\vec{H},\omega_4\right>+\phi''(x_4)| \omega_4^T| ^2\\
       &\left<\nabla\Phi,\nabla H\right>=\phi'(x_4)\left<\nabla x_4,\nabla H\right>=\phi'(x_4)\left<\omega_4^{T},\nabla H\right>\\
       &| \nabla \Phi| ^2=\phi'(x_4)^2| \nabla x_4| ^2=\phi'(x_4)^2| \omega_4^T| ^2
     \end{align*}

      Here $x_4$ is a short hand of $\left<x,\omega_4\right>$ . Also $\omega_4^T$ means the projection onto the tangent plane of $M_t$.
      In the second line we used the identity $\Delta x=\vec{H}$.

     Notice that $\left<\vec{H},\omega_4\right>=0$ and $\left<\omega_4,\nabla H\right>=0$ on $\kappa^{-1}\Sigma_t\times\mathbb{R}$. By approximation described in step 1 we have:
     $| \omega_4^T-\omega_4| +| \left<\nabla H,\omega_4\right>|   + | \left<\vec{H},\omega_4\right>|  \leq C(L_1)\varepsilon_1$ in $\hat{\mathcal{P}}(\bar{x},-1,L_1,L_1^2)$, since these quantities only involve at most two derivatives of the graph function, which is $\varepsilon_1$ small in $C^{10}$.

     Thus we have the estimate:
     \begin{align*}
       | \partial_t\Phi|  & \leq C(L_1)\varepsilon_1| \phi'(x_4)| \\
       | \Delta\Phi|   & \leq C(L_1)\varepsilon_1| \phi'(x_4)| +| \phi''(x_4)| \\
       | \left<\nabla\Phi,\nabla H\right>|  & \leq  C(L_1)\varepsilon_1| \phi'(x_4)| \\
       | \nabla \Phi| ^2&\leq\phi'(x_4)^2
     \end{align*}
     Recall that $D_j=2^{\frac{j}{100}}\Lambda_1$, $W_j=T_j=2^{\frac{j}{50}}\Lambda_1^2$. Since $\Omega_j\subset\hat{\mathcal{P}}(\bar{x},-1,L_1,L_1^2)$,  there holds $D_j<2L_1$.  We also choose  $\varepsilon_1<D_j^{-2}$.

     Next we choose an even function $\phi\in C^2(\mathbb{R})$ satisfying the following
     \begin{align} \label{}
       | \phi'|  & \leq \frac{1}{20}D_j^{-1/2}   \label{phi condition1}\\
       | \phi''|  & \leq \frac{1}{400}D_j^{-1}   \label{phi condition2}\\
       \phi(W_j)& \geq \ln\left((W_j+D_j+T_j)^{20}\right)
       \label{phi condition3}\\
       \phi(200) & \leq \ln\left(W_j+D_j+T_j\right)  \label{phi condition4}\\
       \phi(0) &= 0, \ \phi'>0 \text{ when } x>0 \label{phi condition5}
     \end{align}
     provided that $j\geq j_1$ for some $j_1$ that only depends on $L_0$.

     We can take
     \begin{align}\label{definition of phi}
       \phi(s)=\frac{1}{400D_j}\ln\cosh(s)
     \end{align}

     It's straightforward to check 
     (\ref{phi condition1}), (\ref{phi condition2}),  (\ref{phi condition4}), (\ref{phi condition5}). We only check (\ref{phi condition3}), notice that $s-1<\ln\cosh(s)\leq s$ :
     \begin{align*}
       \phi(W_j)& \geq \frac{1}{400}D_j^{-1}(W_j-1) \geq 2^{j/200} \geq \ln(2^{j}) \geq \ln((W_j+D_j+T_j)^{20})
     \end{align*}
    as long as $j\geq j_1$.

     Finally we let  $\lambda=\frac{1}{1000}D_j^{-1}, \mu=\frac{1}{8}D_j^{-1/2}$  .

     Note that on the standard $n$ dimensional Bowl soliton we have
     $H\approx \sqrt{\frac{n-1}{2}}d^{-1/2}$ and $| A| ^2 \approx \frac{1}{2}d^{-1}$ when $d$ is large, where $d$ is the distance to the tip.
     In our case $n=2$.

     By approximation described in step 1, we know that
     $| A| ^2>\frac{1}{8}D_j^{-1}$ and $\frac{1}{4}D_j^{-1/2}<H<4(\kappa D_j)^{1/2}$ in $\Omega_j$ since the curvature only involves two derivatives of the graph function. In particular $H-\mu>\frac{1}{8}D_j^{-1/2}$ and therefore the denominator is never 0.

     With above discussions we show that the coefficient of $f$ in (\ref{EQNforf}) is negative:
     \begin{align*}
       &\lambda-\frac{\mu| A| ^2}{H-\mu}-\partial_t\Phi+\Delta \Phi+|
       \nabla\Phi| ^2+2\frac{\left<\nabla \Phi,\nabla H\right>}{H-\mu}\\
       \leq &\lambda-\frac{\mu H^2}{3(H-\mu)} +C(L_1)\varepsilon_1| \phi'(x_4)| +| \phi''(x_4)| +\phi'(x_4)^2+C(L_1)\varepsilon_1| \phi'(x_4)| \cdot 8D_j^{1/2} \\
       \leq & -\frac{1}{96}D_j^{-1}+\frac{1}{1000}D_j^{-1} +\frac{1}{400}D_j^{-1}+\frac{1}{400}D_j^{-1} + C(L_1)\varepsilon_1<0
     \end{align*}
     in the second line we used $| A| ^2\geq H^2/3$ and in the third line we used (\ref{phi condition1}) (\ref{phi condition2})  .
     We will choose $\varepsilon_1$ sufficiently small after fixing a large $j$ and $L_1$, therefore we may assume that $ C(L_1)\varepsilon_1 < \frac{1}{1000}D_j^{-1}$, which justified the last inequality.

     Then maximum principle applies to (\ref{EQNforf}), we have
     $\sup\limits_{\Omega_j}| f| \leq \sup\limits_{\partial\Omega_j}| f| $.

     The boundary $\partial\Omega_j=\partial^1\Omega_j\cup\partial^2\Omega_j\cup\Omega_j^{-1-T_j}$.
     By the claim in the Step 5 and (\ref{phi condition1})- (\ref{phi condition5}), we can estimate $| f| $ on each of the boundary portion:
     \begin{itemize}
       \item on the boundary portion $\partial^1\Omega_j$:
            \begin{align*}
              \sup\limits_{\partial^1\Omega_j}| f|  &\leq  \sup\limits_{\partial^1\Omega_j}\frac{| u| H}{(H-\mu)H}\\
              &\leq C_1(W_j+D_j+T_j)^2 2^{-j}\varepsilon\cdot 32D_j \\
              &\leq 2^{-j/2}C\varepsilon
            \end{align*}

       \item on the boundary portion $\partial^2\Omega_j$:
            \begin{align*}
              \sup\limits_{\partial^2\Omega_j}| f| \leq & \sup\limits_{\partial^2\Omega_j}e^{-\phi(W_j)}\frac{| u| H}{H(H-\mu)}\\
              \leq &e^{-\phi(W_j)}\cdot C_1(W_j+D_j+T_j)^2 \varepsilon\cdot32D_j \\
               \leq& C(W_j+D_j+T_j)^{-10}\varepsilon\\
               \leq& 2^{-j/5}C\varepsilon
            \end{align*}
         we used (\ref{phi condition3}) here.
       \item on the boundary portion $\Omega_j^{-1-T_j}$:
            \begin{align*}
              \sup\limits_{\Omega_j^{-1-T_j}}| f| \leq & e^{-\lambda T_j}\sup\limits_{\Omega_j}\frac{| u| H}{H(H-\mu)} \\
               \leq & e^{\frac{-D_j}{1000}}\cdot C_1(W_j+D_j+T_j)^2\cdot 32D_j\varepsilon \\
               \leq & 2^{-j}C\varepsilon
            \end{align*}
       we used $\lambda=\frac{1}{1000}D_j^{-1}$ and $T_j=D_j^2$ here.
     \end{itemize}
      $C$ is a constant that only depends on $L_0$.

     Putting them together we have:
     \begin{align}\label{maxf}
        \sup\limits_{\Omega_j}| f| \leq &\sup\limits_{\partial\Omega_j}| f| \leq
         2^{-\frac{j}{5}}C\varepsilon
     \end{align}
    provided that $j\geq j_1$.

     Now we come back to look at the parabolic neighborhood $\hat{\mathcal{P}}(\bar{x},-1,100,100^2)$. By normalization $H(\bar{x},-1)=1$ and approximation, in this parabolic neighborhood we have:
     \begin{itemize}
       \item $| x_4| <200$
       \item $-10^4-1\leq t\leq -1$
       \item $d(x,l_t)<10^5\kappa+\Lambda_{\star}$
     \end{itemize}
      $| x_4| <200, -10^4-1\leq t\leq -1, d(x,l_t)<10^5\kappa+\Lambda_{\star}$. When $L_1$ is large enough and $j$ is appropriate such that
      \begin{align}\label{j condition A}
        \hat{\mathcal{P}}(\bar{x},-1,100,100^2)\subset\Omega_j\subset\text{Int}(j)
      \end{align}
     Then we have:
     \begin{align}\label{finalineqn}
       | u(y,s)| H(y,s)= & e^{\phi(x_4)-\lambda(t+1)}(H(y,s)-\mu)| f(y,s)| H(y,s)\\
       \leq & C e^{\phi(200)+10^4\lambda}\cdot D_j\cdot 2^{-j/5}\varepsilon\notag \\
       \leq & C(W_j+D_j+T_j)D_j\cdot 2^{-j/5}\varepsilon\notag\\
       \leq & 2^{-j/10}C_2\varepsilon \notag
     \end{align}
     in $\hat{\mathcal{P}}(\bar{x},-1,100,100^2)$,
     where $C_2$ depends only on $L_0,\varepsilon_0$.

     We pick the constants in the following order: First we can find $j=j_2\geq j_1$ depending only on $L_0,\varepsilon_0$ such that (\ref{j condition A}) holds and
     \begin{align}\label{j condition 2}
       2^{-j_2/10}C_2<\frac{1}{2}
     \end{align}
     Next we pick $L_1$ large enough and finally pick $\varepsilon_1$ small enough
     (Therefore, $L_1$ depends on $L_0, \varepsilon_0$ and $\varepsilon_1$ depends
      on $L_0, \varepsilon_0, L_1$)

     With such choice of constants, the theorem is proved by combining
      (\ref{finalineqn}), (\ref{j condition 2}), (\ref{KHmax}) and the definition.

\end{proof}

    \section{Canonical neighborhood Lemmas and the proof of the main theorem}\label{section proof of main theorem}

    The goal of this section is to establish the canonical neighborhood Lemmas and then prove Theorem \ref{main theorem}.

    Recall that the translating soliton satisfies the equation $ H = \left<V,\nu\right>$ for some fixed nonzero vector $V$, where $\nu$ is the inward pointing normal vector.
    With a rotation and dilation we may assume that $V=\omega_3$ is a unit vector. The equation then becomes:
\begin{align}\label{translatoreqn1}
  H &= \left<\omega_3,\nu\right>
\end{align}
  $M+t\omega_3$ is an eternal solution to the mean curvature flow . It's useful to consider the parabolic picture. Therefore  we denote $M_t=M+t\omega_3$ to be the mean curvature flow solution associated with $M$.
    Define the height function in the space time to be
\begin{align}\label{height def}
  h(x,t)=\left<x,\omega_3\right>-t
\end{align}

  Throughout this section the mean curvature flow solution is always assumed to be embedded and complete.

\begin{Def}
  For any ancient solution $M_t$ defined on $t<0$, a blow-down limit, or a tangent flow at $-\infty$, is the limit flow of $M^j_t=c_j^{-1}M_{c_j^2t}$ for some sequence $c_j\rightarrow\infty$, if the limit exists.
\end{Def}

    By the previous work 
    (cf \cite[Theorem 1.11]{haslhofer2017mean}, \cite{white2003nature, white2000size}, \cite{sheng2009singularity}, \cite{huisken1999convexity}) : if $M$ is mean convex and noncollapsed ancient solution, then any blow-down sequence $c_j^{-1}M_{c_j^2t}$ has a subsequence that converges smoothly to $S^k_{\sqrt{-2kt}}\times\mathbb{R}^{3-k}$ for some $k=0,1,2,3$ with possibly a rotation. In particular at least one blow-down limit exists.

    The Guassian density of the hypersurface is given by
    \begin{align*}
      \Theta_{x_0,t_0}(M)=\int_{M}\frac{1}{(4\pi t_0)^{\frac{n}{2}}}e^{-\frac{| x-x_0| ^2}{4t_0}}d\mu
    \end{align*}

    By Huisken's monotoncity formula \cite{huisken1990asymptotic}, we have
    the following:

    \begin{Lemma}\label{unique blowdown}
        For any mean convex and noncollapsed ancient solution $M_t$ defined for $t<0$, suppose that $M^{\infty}_t$ is a blow-down limit, then up to
        a rotation $M_t^{\infty}$ is independent of the blow-down sequence.
    \end{Lemma}
    \begin{proof}
      Suppose that $M^{\infty}_t$ is the limiting flow of $M^j_t=c_j^{-1}M_{c_j^2t}$ for some sequence $c_j\rightarrow\infty$. By the previous discussion, $M^{\infty}_t$ must be one of the self-similar generalized cylinders $S^k_{\sqrt{-2kt}}\times\mathbb{R}^{3-k}$.
      Moreover, mean convex ancient solution must be convex by the convexity estimate \cite{huisken1999convexity}
      (also cf \cite[Theorem 1.9]{haslhofer2017mean}). Therefore the convergence is smooth with multiplicity one.
      By Huisken's monotonicity formula \cite{huisken1990asymptotic},
      $\Theta_{x_0,t_0+s_0-t}(M_{t})$  is monotone increasing in $t$.
      Hence \[\Theta = \lim\limits_{t\rightarrow\infty} \Theta_{x_0,t_0+s_0-t}(M_{t})\leq \infty\] exists.

      Since $\Theta$ is scaling invariant we have:
      \begin{align}\label{entropy bound 1}
        \Theta_{c_j^{-1}x_0,c_j^{-2}(t_0+s_0)-t}(c_j^{-1}M_{c_j^2t})=\Theta_{x_0,t_0+s_0-c_j^2t}(M_{c_j^2t})
      \end{align}
      whenever $c_j^2t<s_0$.

      By convexity we have
      \begin{align}\label{Area bound}
         \text{Area}(M_t\cap B_R(0))\leq CR^3
      \end{align} for some uniform constant $C$.

      Taking $t=-1$ in (\ref{entropy bound 1}).
      Since $c_j^{-1}x_0\rightarrow 0, c^{-2}_j(t_0+s_0)-t\rightarrow 1$ and $c_j^{-1}M_{-c_j^2}$ converges smoothly to $M^{\infty}_{-1}$
      as $j\rightarrow \infty$, together with (\ref{Area bound}) and the exponential decay of the Gaussian weight, we have
      \begin{align*}
       \Theta_{0,1}(M^{\infty}_{-1})=\lim\limits_{j\rightarrow\infty}\Theta_{c_j^{-1}x_0,c_j^{-2}(t_0+s_0)-t}(c_j^{-1}M_{c_j^2t})=\Theta
      \end{align*}
      Since $\Theta_{0,1}(S^{k}\times\mathbb{R}^{3-k})$, $k=0,1,...,3$ are all different numbers,  $M_{-1}^{\infty}$ must have the same shape,
      i.e  they must be the same up to rotation.

    \end{proof}
   \begin{Rem}
        By the work of Colding--Minicozzi \cite{colding2015uniqueness}, the blow-down limit is indeed unique without any rotation, but we don't need this strong result here.
      \end{Rem}

    In particular if $M_t$ is a translating solution, we can interpret the blow down process in a single time slice:
    \begin{Cor}\label{blow down for translator}
      Given $M^3\subset\mathbb{R}^{4}$ a strictly mean convex, noncollapsed translator which satisfies (\ref{translatoreqn1}), then for any $R>1, \varepsilon>0$, there exists a large $c_0$ such that, if $a\geq c_0$ then $a^{-1}(M-a^2\omega_3)\cap B_{R}(0))$ is $\varepsilon$ close to  a $S^k_{\sqrt{2k}}\times\mathbb{R}^{3-k} $ with some rotation for a fixed $1\leq k\leq 2$ in $C^{10}$ norm.
    \end{Cor}

    In the following we prove some canonical neighborhood Lemmas.
    \begin{Lemma}\label{canonical nbhd lemma}
      Let $M^3\subset\mathbb{R}^4$  be a  noncollapsed, convex, smooth translating soliton of mean curvature flow.\  \ $p\in M$ is a fixed point. Set $M_t$ to be the associated translating solution of the mean curvature flow. Suppose that one blow-down limit of $M_t$ is the shrinking $S^1\times\mathbb{R}^2$ and that the sub-level set $M_t\cap\{h(\cdot,t)\leq h_0\}$ is compact for any $h_0\in\mathbb{R}$.
      Given  $L>10, \varepsilon>0$,  there exist a constant $\Lambda$ with the following property:
      If $x\in M$ satisfies $| x-p| \geq\Lambda$, then after a parabolic rescaling by the factor  $H(x,t)$,
     the parabolic neighborhood $\hat{\mathcal{P}}(x,t,L,L^2)$ is $\varepsilon$ close to the corresponding piece of the shrinking $S^{1}\times \mathbb{R}^2$, or the translating $\text{Bowl}^{2}\times \mathbb{R}$.

    \end{Lemma}

    \begin{proof}
      If $M$ is not strictly mean convex, then by the maximum principle $M$ must be flat plane, therefore all the blow-down limit is flat plane, a contradiction.

      Without loss of generality we assume that $p$ is the origin, and that $\omega_3$ is the unit vector in $x_3$ axis which points to the positive part of $x_3$ axis.

    Argue by contradiction. Suppose that the conclusion is not true, then there exist a sequence of points $x_j\in M$ satisfying $| x_j-p| \geq j$ but $\hat{\mathcal{P}}(x_j,0,L,L^2)$ is not $\varepsilon$ close to either one of the models after the appropriate parabolic rescaling.

  By the long range curvature estimate (ie bounded rescaled distance implies bounded rescaled curvature, cf \cite{white2000size,white2003nature}, \cite[Corollary 3.2]{haslhofer2017mean},  $H(p)| x_j-p| \rightarrow \infty$ implies
  \begin{align}\label{long range curvature estimate}
    H(x_j)| x_j-p| \rightarrow \infty
  \end{align}

  After passing to a subsequence, we may assume
  \begin{align}\label{distancedoubling}
    | p-x_{j+1}| \geq 2| p-x_j|
  \end{align}

  Now let $M^{(j)}_t=H(x_j)\Big(M_{H^{-2}(x_j)t}-x_j\Big)$ be a sequence of rescaled solutions. Under this rescaling $M^{(j)}_t$ remains eternal, $x_j$ is sent to the origin and $H_{M_0^{(j)}}(0)=1$.
  By the global convergence theorem (cf \cite[Theorem 1.10]{haslhofer2017mean})  $M^{(j)}_t$ converges to an ancient solution $M^{\infty}_t$ which is  smooth,  non-collapsed, weakly convex, and $H_{M^{\infty}_0}(0)=1$. Thus $M^{\infty}_0$ is strictly mean convex.
  Denote by $K^{\infty}$ the convex domain bounded by $M^{\infty}$.

  Suppose that $p$ is sent to $p^{(j)}$ in the $j^{th}$ rescaling. Then  $p^{(j)}=-H(x_j)x_j$ and $| p^{(j)}| \rightarrow \infty$ by (\ref{long range curvature estimate}).
  After passing to a subsequence, we may assume that $\frac{p^{(j)}}{|
  p^{(j)}| }=-\frac{x_j}{| x_j| }\rightarrow \Theta \in S^3$.
  Suppose that $l$ is the line in the direction of $\Theta$, i.e. $l=\{s\Theta\ |  \ s\in\mathbb{R}\}$.

  Since rescaling doesn't change angle, we have $\angle x_jpx_{j+1}\rightarrow 0$. \\

  \noindent\textbf{Claim: } $\angle x_{j+1}x_jp\rightarrow \pi$

   To see this, we first compute $| x_j-x_{j+1}| $:
  \begin{align*}
    | x_j-x_{j+1}| ^2=| x_j-p| ^2+| x_{j+1}-p| ^2-2| x_j-p|  | x_{j+1}-p| \cos(\angle x_jpx_{j+1})
  \end{align*}
  note that $1/2\leq \cos(\angle x_jpx_{j+1})\leq 1$ for large $j$, together with the (\ref{distancedoubling}) we have
  \begin{align}\label{edgerelation1}
    \frac{1}{2}| x_{j+1}-p| <| x_{j+1}-p| -| x_j-p| <| x_j-x_{j+1}|  < | x_{j+1}-p|
  \end{align}
  Therefore $| px_{j+1}| $ is the largest edge in the triangle $px_jx_{j+1}$,  consequently $\angle px_jx_{j+1}$ is the largest and is no less than $\pi/3$. On the other hand we can use the Sine law in the triangle:
  \begin{align*}
    \frac{\sin(\angle x_{j+1}x_jp)}{| x_{j+1}-p|
    }=\frac{\sin(\angle x_jpx_{j+1})}{| x_j-x_{j+1}| }
  \end{align*}
  Together with (\ref{edgerelation1}) we will have  $\sin(\angle x_{j+1}x_jp)\leq \sin(\angle x_jpx_{j+1})\rightarrow 0$. Therefore $\angle x_{j+1}x_jp\rightarrow \pi$. Hence the claim is proved.

  Looking back to the rescaled picture, $x_{j+1}$ is sent to $H(x_j)(x_{j+1}-x_j)$. By (\ref{edgerelation1}) $H(x_j)| x_{j+1}-x_j| \geq H(x_j)| x_j-p| =|
  p^{(j)}| \rightarrow \infty$.
  Together the angle convergence, the limit contains the  line $l$. Thus it can't be strictly convex.
  By the strong maximum principle (cf \cite[Section 4]{hamilton1986four}, \cite[Appendix]{white2003nature}), \cite[Appendix A]{haslhofer2017mean}
   $M_t^{\infty}$ split off a line. That is,
  $M ^{\infty}_t=\mathbb{R}\times M'_t$ where $\mathbb{R}$ is in the direction of $l$ and $M'_t$ is an ancient solution that is non-collapsed, strictly mean convex, uniformly 2-convex. \\

  \noindent\textbf{Case 1: } If $M'_t$ is noncompact, then by the classification result of Brendle--Choi \cite{brendle2019uniqueness}, we conclude that, up to translation and rotation,  $M'_t$ is  the shrinking $S^1\times\mathbb{R}$, or the translating $\text{Bowl}^{2}$.
    This means that for large $j$, the parabolic neighborhood $\hat{\mathcal{P}}^{(j)}(0,0,L,L^2)$ is $\varepsilon$ close to either $S^1\times\mathbb{R}^2$ or $\text{Bowl}^2\times\mathbb{R}$, a contradiction.\\

    \noindent\textbf{Case 2: } If $M'_t$ is compact,  then $l$ is the $x_3$ axis, consequently $\frac{x_j}{| x_j| }$ converges to $\omega_3$.
    In fact, let's denote  by $l^{\perp}$ the plane perpendicular to $l$ that passes through the origin. (Note that $l$ also passes through the origin by definition).
    Hence $l^{\perp}\cap M_t^{\infty}$ is isometric to $M_t'$.
    Suppose that $l$ is not the $x_3$ axis, then there is a unit vector $v\perp l$ such that $\left<v,\omega_3\right><0$.
    Since $M_t^{\infty}$ splits off in the direction $l$ and  each cross section is a closed surface $M_t'$, we can find a point $y_j'\in M_0^{\infty}\cap l^{\perp}$ such that the inward normal vector $\nu$ of $M_0^{\infty}$ at $y_j'$ coincide with $v$.
    By the convergence, for large $j$ there is a point $y_j\in M_0^{(j)}\cap l^{\perp}$ such that the unit inward normal $\nu_{M^{(j)}_0}(y_j)$ is sufficiently close to $v$.  But this implies that $H_{M^{(j)}_0}(y_j)=H_{M_0}(x_j)^{-1}\left<\nu_{M^{(j)}_0}(y_j),\omega_3\right><0$ a contradiction.

    We know by convergence that, for each $R>0$,   $M^{(j)}_0$ is close to $\mathbb{R}\times M_0'$ in the ball $B_R(0)$ whenever $j$ is large. By the above argument we know that the $\mathbb{R}$ direction is parallel to $\omega_3$.
    Hence the cross section perpendicular to the $\mathbb{R}$ factor are the level set of the height function.
    In particular we can choose $R\gg \text{diam}(M_0')$ and consequently the cross section $M_0^{(j)}\cap \{h(\cdot,0)=0\}$  converges smoothly to the cross section $(\mathbb{R}\times M_0')\cap \{h(\cdot,0)=0\}$, which is $M_0'$.

     Now let's convert it to the unrescaled picture. Let $h_j=h(x_j,0)$ and $N_j=M\cap \{h(\cdot,0)=h_j\}$ ($N_j$ should be considered as  a hypersurface in $\mathbb{R}^3$). Then $N_j$, after appropriate rescaling, converges to $M_0'$.

    Define a scale invariant quantity for hypersurface(possibly with boundary):
    \begin{align*}
      \text{ecc}(N_j)=\text{diam}(N_j)\sup H_{N_j}
    \end{align*}
     where diam denotes the extrinsic diameter and $H_{N_j}$ is the mean curvature of $N_j$ in $\mathbb{R}^3$. Clearly ecc is monotone in the sense that $\text{ecc}(N)\geq \text{ecc}(N')$ whenever $N'\subset N$.
    By convergence, $\text{ecc}(N_j)\rightarrow\text{ecc}(M_0')$, in particular:
    \begin{align}\label{ecc M0'}
      \text{ecc}(N_j)< 2\text{ecc}(M_0')
    \end{align}
    for all large $j$.

    On the other hand, by assumption each sub-level set $\{h(\cdot, 0)\leq h_0\}$ is compact, this means $h_j\rightarrow+\infty$.
    By Corollary \ref{blow down for translator},
    for any $\eta>0, R>1$ there exists $A_j\in SO(4)$ for each sufficiently large $j$ such that $\sqrt{h_j}^{-1}(M_0-h_j\omega_3)\cap B_R(0)$ is $\eta$ close to $\Sigma_j:=A_j(S^1_{\sqrt{2}}\times\mathbb{R}^2)$, which is a rotation of $S^1_{\sqrt{2}}\times\mathbb{R}^2$.
    Since $\left<\nu,\omega_3\right>>0$ on $\sqrt{h_j}^{-1}(M_0-h_j\omega_3)\cap B_R(0)$, by approximation $\left<\nu,\omega_3\right>>-C\eta$ on $\Sigma_j\cap B_R(0)$.
    If $R\gg\sqrt{2}$ then we are able to take antipodal point of $S^1$ and thus $| \left<\nu,\omega_3\right>| <C\eta$ on $\Sigma_j\cap B_R(0)$. This means the $\mathbb{R}^2$ factor of $\Sigma_j$ must be almost perpendicular to the level set of the height function $h$.
    Consequently, the intersection of $\Sigma_j\cap B_R(0)$ with $\{h(\cdot,0)=0\}$ must be $C\eta$ close to some rotated $(S^1_{\sqrt{2}}\times\mathbb{R})\cap B_R(0)$, hence the mean curvature (computed in $\mathbb{R}^3$) is at least $\frac{1}{\sqrt{2}}-C\eta$. (See Appendix for more detailed explanation).

    Now let $\tilde{N}_j=\sqrt{h_j}^{-1}(M_0-h_j\omega_3)\cap\{h(\cdot,0)=0\}\cap B_R(0)$. Then $\tilde{N}_j$, considered as surface in $\mathbb{R}^3$, is $C\eta$ close to $\Sigma_j\cap\{h(\cdot,0)=0\}\cap B_R(0)$. Therefore
    \begin{align}\label{ecc cylinder}
      \text{ecc}(\tilde{N}_j)\geq (\frac{1}{\sqrt{2}}-C\eta)R
    \end{align}

    Now we take $\eta<\frac{1}{8C}$, $R>4\text{ecc}(M_0')$ and $j$ large accordingly. Note that $\tilde{N}_j=\sqrt{h_j}^{-1}(N_j-h_j\omega_3)$, then the scale invariance of ecc together with (\ref{ecc M0'}) and (\ref{ecc cylinder})  lead to a contradiction.

    \end{proof}

    \begin{Lemma}\label{blowdownnecklemma}
      Suppose that $M^3\subset\mathbb{R}^4$ is  noncollapsed, strictly convex translating soliton of mean curvature flow.
      Set $M_t$ to be the associated translating solution. Suppose that one blow-down limit of $M_t$ is the shrinking $S^2\times\mathbb{R}$  and that the sub-level set $M_t\cap\{h(\cdot,t)\leq h_0\}$ is compact for any $h_0\in\mathbb{R}$. Then $M=\text{Bowl}^3$.
    \end{Lemma}
    \begin{proof}
    Let's fix $p\in M$. Assume without loss of generality  that $p$ is the origin and  $\omega_3$ is the unit vector in $x_3$ axis which points to the positive part of $x_3$ axis.

    Using 
    the argument from the Case 2 in Lemma \ref{canonical nbhd lemma}, the ball $B_{100\sqrt{a}}(a\omega_3)\cap M$ is close a $S^2_{\sqrt{2a}}\times\mathbb{R} $ with some translation and rotation for all large $a$.
    Moreover, the fact that $\left<\nu, \omega_3\right> >0$ and approximation imply that the $\mathbb{R}$  direction must be almost parallel to the $x_3$ axis, and the cross section $\{h(\cdot,0)=a\}$ is close to  $S^2$ after some scaling, for all sufficiently large $a$.
    In particular, this means that $M\cap \{h(\cdot,0)\geq A\}$ is uniformly 2-convex for a fixed large number $A$.
    Together with the assumption that the sub-level set $\{h(\cdot, 0)\leq A\}$ is compact and the assumption that $M$ is strictly convex, we know that $M$ is uniformly 2-convex. Clearly $M$ is noncompact, thus the classification result of Haslhofer \cite{haslhofer2015uniqueness} or Brendle--Choi \cite{brendle2019uniqueness} applies and $M=\text{Bowl}^3$.

    \end{proof}

\begin{proof}[Proof of Theorem \ref{main theorem}]

      Let $M_t=M+t\omega_3$ be the associated translating mean curvature flow solution. Denote by $g$ the metric on $M$ and $g_t$ the metric on $M_t$.

      If $M$ is not strictly mean convex, then by strong maximum principle mean curvature vanishes everywhere, therefore $M$ is flat. The conclusion then follows trivially.
      If $M$ is strictly mean convex but is not strictly convex, then by maximum principle \cite{hamilton1986four} it must split off a line, namely, it is a product of a line and a convex, uniformly 2-convex, non-collapsed translating solution (which must be noncompact). By the classification result  \cite{haslhofer2015uniqueness} or \cite{brendle2019uniqueness} , $M=$Bowl$^2\times\mathbb{R}$  and we are done.

      We may then assume that $M$ (or $M_t$) is strictly convex. It suffices to find a normalized rotation vector field that is tangential to $M$.

      Suppose that $M_t$ bounds the open domain $K_t=K+t\omega_3$ and attains maximal mean curvature at the tip $p_t=p+t\omega_3$.
      After some rotation and  translation we assume that $p$ is the origin and $\omega_3$ is the unit vector in $x_3$ axis which points to the positive part of $x_3$ axis.

      As in  \cite[page 2398]{haslhofer2015uniqueness}, we take the gradient of (\ref{translatoreqn1}) at the point $p$:
\begin{align*}
  \nabla H &= A(\omega_3^{T})
\end{align*}
where $A$ is the shape operator.

Since $H$ attains maximum at $p$, $\nabla H=0$. By strict convexity $A$ is non-degenerate, so $\nabla H=A(\omega_3^T)=0 \Rightarrow \omega_3^T=0$. Hence $\vec{H}=\omega_3$. In particular $\nu(p)=\omega_3$ and $H(p)=1$, thus $p_t$ is indeed the trajectory that moves by mean curvature.

      By the convexity, $K$ is contained  in the upper half space $\{x_3\geq 0\}$. The strict convexity implies that there is a cone $\mathcal{C}_{\eta}=\{x\in\mathbb{R}^4|  x_3\geq \eta\sqrt{x_1^2+x_2^2+x_4^2}\}$ ($0<\eta<1$) such that
      \begin{align}\label{m contained in cone}
        K_t\backslash B_{1}(p)\subset \mathcal{C}_{\eta}\backslash B_{1}(p)
      \end{align}

      The height function  $h(x,t)=\left<x,\omega_3\right>-t$ \ then measures the signed distance from the support plane $\{x_3=t\}$ at the tip.
      (\ref{m contained in cone}) implies that any sub-level set $\{h(\cdot, 0)\leq h_0 \}$ is compact.

       By Corollary \ref{unique blowdown}, a blow-down limit of $M_t$ exists and each of possible blow-down limit must be  $S^{k}_{\sqrt{-2kt}}\times\mathbb{R}^{3-k}$ up to rotation,
       for a fixed $k=0,1,2,3$. If $k=3$, then $M$ is compact, thus can't be a translator. If $k=0$, by Huisken's monotonicity formula \cite{huisken1990asymptotic} $M$ must be a flat plane, thus the result follows trivially.
      If $k=2$, then by Lemma \ref{blowdownnecklemma}, $M$ is Bowl$^3$, and the result follows immediately.
      So from now on we assume that  $k=1$ and Lemma \ref{canonical nbhd lemma} is applicable.

      By Lemma \ref{canonical nbhd lemma}, there exists $\Lambda'$ depending on $L_1,\varepsilon_1$ given in the Theorem \ref{bowl x R improvement} such that if $(x,t)\in M_t$ satisfies $h(x,t)\geq \Lambda'$, then $\hat{\mathcal{P}}(x,t,L_1,L_1^2)$ is $\varepsilon_1$ close to either a piece of translating Bowl$^2\times\mathbb{R}$ or a family of shrinking cylinder $S^1\times\mathbb{R}^2$ after rescaling.

      By the equation (\ref{translatoreqn1}), $h$ is nonincreasing. In fact, $\partial_t h(x,t)=\left<\nu,\omega_3\right>-1\leq 0$.

      Using (\ref{m contained in cone}) we have $h(x,t)\geq \frac{\eta}{2}| x-p_t| $ whenever $| x-p_t| \geq 1$. Also it's obvious that $h(x,t)\leq| x-p_t| $.

     By the long range curvature estimate (cf \cite{white2000size, white2003nature}, \cite[Corollary 3.2]{haslhofer2017mean}) there exists  $\Lambda>\Lambda'$ such that
     $H(x,t)| x-p_t| >2\cdot10^3\eta^{-1}L_1$ for all $(x,t)$ satisfying $h(x,t)>\Lambda$. Consequently $H(x,t)h(x,t)>10^3L_1$.

     Define

     $\Omega_j=\{(x,t)\ | \ x\in M_t,\ t\in[-2^{\frac{j}{100}}, 0],\ h(x,t)\leq2^{\frac{j}{100}}\Lambda \}$

     $\partial^{1} \Omega_j = \{(x,t)\ | \ x\in M_t,\ t\in[-2^{\frac{j}{100}}, 0],\ h(x,t)=2^{\frac{j}{100}}\Lambda \}$

     $\partial^{2} \Omega_j = \{(x,t)\in\partial\Omega_j \ | \ t=-2^{\frac{j}{25}} \}$

      \noindent\textbf{Step 1:} If $x\in M_t$ satisfies $h(x,t)\geq 2^{\frac{j}{100}}\Lambda$, then $(x,t)$ is $2^{-j}\varepsilon_1$ symmetric.

      When $j=0$ the statement is true by the choice of $\Lambda', \Lambda$. If the statement is true for $j-1$. Given $x\in M_t$ satisfying $h(x,t)\geq 2^{\frac{j}{100}}\Lambda$, the choice of $\Lambda$ ensures that
      $L_1H(x,t)^{-1}<10^{-3}h(x,t)$. So every point in the geodesic ball $B_{g_t}(x,L_1H(x)^{-1})$ has height at least $(1-10^{-3})h(x,t)\geq 2^{\frac{j-1}{100}}\Lambda$. Moreover, the height function is nonincreasing in time $t$, so we conclude that  the parabolic neighborhood $\hat{\mathcal{P}}(x,t,L_1,L_1^2)$ is contained in the set $\{(x,t) |  x\in M_t, h(x,t)\geq 2^{\frac{j-1}{100}}\Lambda\}$. In particular every point in
      $\hat{\mathcal{P}}(x,t,L_1,L_1^2)$ is $2^{-j+1}\varepsilon_1$ symmetric by induction hypothesis.
      Now we can apply either Lemma \ref{Cylindrical improvement} or Lemma \ref{bowl x R improvement} to obtain that $(x,t)$ is $2^{-j}\varepsilon$ symmetric.\\

      \noindent\textbf{Step 2:} The intrinsic diameter of the set $M_t\cap\{h(x,t)=a\}$ is bounded by $6\eta^{-1}a$ for $a\geq 1$.

      It suffices to consider $t=0$. Since $M$ is convex, the level set $\{h(x,t)=a\}$ is also convex. By (\ref{m contained in cone}), $M \cap\{h(x,0)=a\}$ is contained in a 3 dimensional ball of radius $\eta^{-1}a$, thus has extrinsic diameter at most $2\eta^{-1}a$. By Appendix \ref{intrisic diam extrinsic diam appendix}, the intrinsic diameter of a convex set is bounded by {the}  triple of the extrinsic diameter, the assertion then follows immediately. \newline

      \noindent\textbf{Step 3:} For each $j$, there is a single normalized rotation vector field $K^{(j)}$ satisfying $| \left<K^{(j)},\nu\right>| H\leq C2^{-\frac{j}{2}}\varepsilon_1$ on $\partial^{1} \Omega_j$

      We fix a point $(\bar{x},\bar{t})$ in $\partial^1\Omega_j$. For any other $(x,t)\in\partial^1\Omega_j$, it's possible to find a sequence of points $(x_k,t_k)$, $k=0,1,...,N$ such that
      \begin{enumerate}
        \item $(x_0,t_0)=(\bar{x},\bar{t}), \ (x_N,t_N)=(x,t)$
        \item $N\leq C2^{\frac{j}{25}}$
        \item $(x_{k+1},t_{k+1})\in\hat{\mathcal{P}}(x_{k},t_{k},\frac{1}{10},\frac{1}{100})$
      \end{enumerate}

      In fact, since the mean curvature is bounded by 1,  we can use at most $C\cdot 2^{\frac{j}{25}}$ consecutive points to reach the time $t$.

      Then keep $t$ fixed, by Step 2 we can use another $C\eta^{-1}2^{\frac{j}{100}}\Lambda$ consecutive points to reach $(x,t)$.

      By Step 1, for each $(x_k,t_k)$ there is a normalized rotation vector field $K^{(x_k,t_k)}$ such that $| \left<K^{(x_k,t_k)}\right>| H\leq 2^{-j}\varepsilon_1$ and $| K^{(x_k,t_k)}| H\leq 5$ in $\hat{\mathcal{P}}(x_k,t_k, 100,100^2)$. The structure of Bowl$\times\mathbb{R}$ and $S^1\times\mathbb{R}^2$ together with Lemma \ref{canonical nbhd lemma} gives that

      $B_{g_{t_{k+1}}}(x_{k+1},10H(x_{k+1},t_{k+1})^{-1})\times
      \{t_{k+1}\}\subset\hat{\mathcal{P}}(x_k,t_k,100,100^2)$.

      Applying Lemma \ref{NRVFDistanceCyl} or Lemma \ref{NRVFDistanceBowl} to $K^{(x_k,t_k)}$ and $K^{(x_{k+1},t_{k+1})}$:
      \begin{align*}
      \min\left\{
     \begin{array}{c}
       \max\limits_{B_{L'H(x_{k+1},t_{k+1})^{-1}}(x_{k+1})}| K^{(x_k,t_k)}-K^{(x_{k+1},t_{k+1})}| H \\
       \max\limits_{B_{L'H(x_{k+1},t_{k+1})^{-1}}(x_{k+1})}| K^{(x_k,t_k)}+K^{(x_{k+1},t_{k+1})}| H
     \end{array}
     \right\}\leq C(L'+1)2^{-j}\varepsilon_1
     \end{align*}

      Without loss of generality we can assume that the first quantity always achieves the minimum. Moreover, we can choose $L'=2\cdot2^{\frac{j}{20}}\Lambda$ so that $B_{L'H(x_{k+1},t_{k+1})^{-1}}(x_{k+1})$ contains each time slice of $\Omega_j$. Then we can sum up the inequalities to obtain:
      \begin{align*}
     \min\left\{
       \max\limits_{\Omega_j}| K^{(\bar{x},\bar{t})}-K^{(x,t)}| H, \ \max\limits_{\Omega_j}| K^{(\bar{x},\bar{t})}+K^{(x,t)}| H
     \right\}
       \leq  CL'N2^{-j}\varepsilon_1\leq C2^{-\frac{j}{2}}\varepsilon_1
     \end{align*}

      Let $K^{(j)}=K^{(\bar{x},\bar{t})}$. Then $| \left<K^{(j)},\nu\right>| H\leq C2^{-\frac{j}{2}}\varepsilon_1$ in $\Omega_j$. \newline

      \noindent\textbf{Step 4:} Finishing the proof.

      Define the function $f^{(j)}$ on $\Omega_j$ to be
      \begin{align}\label{f in the final step}
        f^{(j)}(x,t)=\exp ({\lambda_j t})\frac{\left<K^{(j)},\nu\right>}{H-c_j}
      \end{align}
      where $\lambda_j=2^{-\frac{j}{50}}, c_j=2^{-\frac{j}{100}}$.

      As in \cite{brendle2019uniqueness} or (\ref{EQNforf}), we have
      \begin{align*}
        (\partial_t -\Delta )f^{(j)}
       =&\left(\lambda_j-\frac{c_j| A| ^2}{H-c_j} \right)f^{(j)}+2\left<\nabla f^{(j)},\frac{\nabla H}{H-c_j}\right>
      \end{align*}

      Since $H\geq 10^{3}h^{-1}$ whenever $h>\Lambda$, we have $H>2\cdot 2^{-\frac{j}{100}}=2c_j$ in $\Omega_j$ for all large $j$.

      Hence
      \begin{align*}
        \lambda_j-\frac{c_j| A| ^2}{H-c_j}\leq \lambda_j-\frac{c_jH^2}{3(H-H/2)}
        \leq \lambda_j-\frac{4}{3}c_j^2<0
      \end{align*}

      By the maximum principle,
      \begin{align*}
        \sup\limits_{\Omega_j}| f^{(j)}| \leq & \sup\limits_{\partial\Omega_j}| f^{(j)}| =
        \max\left\{\sup\limits_{\partial^1\Omega_j}| f^{(j)}| ,\sup\limits_{\partial^2\Omega_j}| f^{(j)}| \right\}
      \end{align*}

      Now by Step 3,
      \begin{align*}
        \sup\limits_{\partial^1\Omega_j}| f^{(j)}| \leq \frac{| \left<K^{(j)},\nu\right>H| }{(H-c_j)H}\leq C\frac{2^{-\frac{j}{2}}\varepsilon_1}{2c_j^2}\leq C2^{-\frac{j}{4}}\varepsilon_1
      \end{align*}

      On the other side,  $| \left<K^{(j)},\nu\right>| \leq C2^{\frac{j}{50}}$ in $\Omega_j$. Therefore, for large $j$:
      \begin{align*}
        \sup\limits_{\partial^2\Omega_j}| f^{(j)}| \leq \exp(-2^{\frac{j}{50}})\frac{| \left<K^{(j)},\nu\right>| }{c_j}\leq 2^{-j}
      \end{align*}

      Putting them together we get $| f^{(j)}| \leq 2^{-j/4}$ in $\Omega_j$. This implies that the axis of $K^{(j)}$ (i.e the 0 set of $K^{(j)}$) has a uniform bounded distance from $p$ or equivalently  $| K^{(j)}(0)| \leq C$.
      (If this is not the case, then passing to subsequence we may assume that $| K^{(j)}(0)| \rightarrow \infty$, also $\tilde{K}^{(j)}=\frac{K^{(j)}}{| K^{(j)}(0)| }$ converges locally to a constant vector field $\tilde{K}$ which is tangential to $M_0$ near $0$, hence $M_0$ must be flat, a contradiction.)

      Passing to a subsequence and taking limit, there exists normalized rotation vector field $K$ which is tangential to $M$. This completes the proof.
\end{proof}

\begin{appendix}

\section{ODE for Bowl soliton}\label{ode of bowl}
    We will set the origin to be the tip of the Bowl soliton and the $x_{n+1}$ to be the translating axis.
    Use the parametrization: $\varphi:\mathbb{R}^n\rightarrow\text{Bowl}^n$:
    $\varphi(x)=(x,h(x))$ where $h(x)=\varphi(| x| )$ is a one variable function satisfying the ODE:
    \begin{align*}
      \frac{\varphi''}{1+\varphi'^2}+\frac{(n-1)\varphi'}{r}=1
    \end{align*}
    with initial condition $\varphi(0)=\varphi'(0)=0$.

    Since $\varphi''\geq0$, we have the inequality:
    \begin{align*}
      \varphi''+\frac{(n-1)\varphi'}{r}\geq 1
    \end{align*}
    Using the integrating factor, when $r>0$:
    \begin{align*}
      (r^{n-1}\varphi')'=r^{n-1}\left(\varphi''+\frac{(n-1)\varphi'}{r}\right)\geq r^{n-1}
    \end{align*}
    Integrating from $0$ to $r$ we obtain $\varphi'\geq \frac{r}{n}$ and $\varphi \geq \frac{r^2}{2n}$

    The mean curvature at $(x,h(x))$ is given by
    \begin{align*}
      \frac{1}{\sqrt{1+\varphi'(| x| )^2}}\leq \frac{1}{\sqrt{1+\frac{| x| ^2}{n^2}}}<\frac{n}{| x| }
    \end{align*}

    As an application, if $n=2$ and let $J$ be the antisymmetric matrix whose only nonzero entries are $J_{12}=-J_{21}=1$. The $| J| H<2$ on the Bowl soliton.
\section{Heat Kernel estimate}\label{appendix heat kernel}
    The purpose is to justify (\ref{hkest2}). Let's use a more general notation, that is, we replace $\frac{L_0}{4}$ by $L$. Denote
       \begin{align*}
         &D(x,y,k_1,k_2,\delta_1,\delta_2)=
       \left\lvert (x_1,x_2)-(\delta_1 y_1, \delta_2 y_2)-(1-\delta_1,1-\delta_2)L+(4k_1,4k_2)L\right\rvert
       \end{align*}

        Then
        \begin{align*}
          K_t&(x,y)=-\frac{1}{4\pi t}\sum_{\delta_i\in \{\pm 1\}, k_i\in \mathbb{Z}}(-1)^{-(\delta_1+\delta_2)/2}\cdot e^{-\frac{D^2}{4t}}
        \end{align*}

       We observe that, when $x\in\Omega_{L/25}$ and $y\in\partial\Omega_{L}$,
       \[| D| \geq \frac{| k_1| +| k_2| +1}{2}L\].

       Then
       \[| \partial_{\nu_{y}}e^{-D^2/4t}| \leq \frac{| D| }{2t}
       e^{-D^2/4t}\leq C\frac{(| k_1| +| k_2| +1)L}{t}\exp\left(-\frac{(| k_1| +| k_2| +1)^2L^2}{16t}\right)\]
       The last inequality holds because the function $\lambda e^{-\lambda}$ is decreasing for $\lambda>1$.
       Consequently,
       \begin{align*}
         | \partial_{\nu_y} K_t(x,y)| &\leq \sum_{n=0}^{\infty}\sum_{| k_1| +| k_2| =n, \delta_i\in\{\pm1\}}\frac{1}{4t}| \partial_{\nu_{y}}e^{-\frac{D^2}{4t}}| \\
         &\leq C\sum_{n=0}^{\infty}\sum_{| k_1| +| k_2| =n} \frac{(| k_1| +| k_2| +1)L}{t^2}\exp\left(-\frac{(| k_1| +| k_2| +1)^2L^2}{16t}\right)\\
          & \leq C\sum_{n=1}^{\infty}\frac{n^2L}{t^2}\exp\left(\frac{-n^2L^2}{50t}\right)\exp\left(\frac{-n^2L^2}{50t}\right)\\
         &\leq C\sum_{n=1}^{\infty}\frac{n^2L}{t^2}\frac{2(50t)^2}{(n^2L^2)^2}e^{\frac{-L^2}{50t}} \leq \frac{C}{L^3}e^{\frac{-L^2}{50t}}
       \end{align*}
       where the last inequality used the fact that $e^{-s}< \frac{2}{s^2}$ when $s>0$. Integrating along the boundary and replacing $t$ by $t-\tau$ we get \begin{align*}
         \int_{\partial{\Omega_{L}}}| \partial_{\nu_y}K_{t-\tau}(x,y)| dy\leq \frac{C}{L^2}e^{\frac{-L^2}{50t}}\leq \frac{CL^2}{(t-\tau)^2}e^{-\frac{L^2}{50(t-\tau)}}
       \end{align*}
       The last inequality is because $t-\tau<L^2$.
       Putting $L_0/4$ in place of $L$ and then we get (\ref{hkest2}).

       \section{Intrinsic and extrinsic diameter of a convex hypersurface}\label{intrisic diam extrinsic diam appendix}

       Give a compact, convex set $K\subset\mathbb{R}^n$, the boundary $M=\partial K$. The intrinsic diameter of $M$ is
       \begin{align*}
         d_1(M)=\sup\limits_{x,y\in M}\inf\limits_{\substack{\gamma(0)=x, \gamma(1)=y\\ \gamma \text{ continous }}} L(\gamma)
       \end{align*}
       The extrinsic diameter of $M$ is
       \begin{align*}
         d_2(M)=\sup\limits_{x,y\in M}| x-y|
       \end{align*}
       We show that $d_1(M)\leq 3d_2(M)$.

       Given a two dimensional plane $P$ whose intersection with $M$ contains at least two points. $P\cap K$ is convex.  Let's restrict our attention to $P$.
       Take $x,y\in P\cap M$ that attains $d_2(P\cap M)$.  Without loss of generality we may assume that $P=\mathbb{R}^2$, $x=(-1,0), y=(1,0)$. Thus $d_2(P\cap M)=2$.

       $P\cap K$ is contained in the rectangle
       $R=\{(x,y)\in\mathbb{R}^2\ |  \ | x| \leq 1, | y| \leq 2\}$ by the choice of $x,y$.
       Let $R^{\pm}$ be the upper/lower half of this rectangle, respectively.
       By convexity we see that $M\cap P\cap R^{+}$ is a graph of a concave function on $[-1,1]$, thus the length is bounded by half of the perimeter of $P$, which is $6$. In the same way $M\cap P\cap R^{-}$ has length at most $6$.

       Now let $x',y'$ attains $d_1(M)$. Then we find some $P$ passing through $x',y'$. The above argument shows that there exists a curve $\gamma$ connecting $x',y'$ with $L(\gamma)\leq 3d_2(M\cap P)$.
       Then we have:
       \begin{align*}
         d_1(M)\leq L(\gamma)\leq 3 d_2(M\cap P)\leq 3d_2(M)
       \end{align*}

       \section{Cross section of $S^1\times\mathbb{R}^2$}
       Let $\omega_3$ be the unit vector in the $x_3$ axis and $h(x)=\left<x,\omega_3\right>$ be the height function. Suppose that $A\in SO(4)$, $R\gg 1$ and $\eta$ sufficiently small. Also suppose that $\left<\nu, \omega_3\right> > -C\eta$ on $\Sigma_j\cap B_R(0)$ where $\Sigma_j=A_j(S^1\times\mathbb{R}^2)$. Then the cross section $\Sigma_j\cap\{h=0\}\cap B_R(0)$ is $C\eta$ close to a rotated $(S^1\times\mathbb{R})\cap B_R(0)$.

       \begin{proof}
         Since $R$ is large, all the possible normal vectors of $(S^1\times\mathbb{R})\cap B_R(0)$ should be in the form $(\cos\theta, \sin\theta, 0,0)$ for $\theta\in\mathbb{R}$. So we have $(\cos\theta, \sin\theta, 0,0)\cdot A\cdot(0,0,1,0)^{T}\geq -C\eta$. Replacing $\theta$ by $\theta+\pi$ we get $| (\cos\theta, \sin\theta, 0,0)\cdot A^{T}\cdot(0,0,1,0)^{T}| \leq C\eta$ for all $\theta$. That means $A_{13}=O(\eta)$ and $A_{23}=O(\eta)$.

         Let $I=\begin{bmatrix}
                  1 & 0 & 0 & 0 \\
                  0 & 1 & 0 & 0 \\
                  0 & 0 & 0 & 0 \\
                  0 & 0 & 0 & 0
                \end{bmatrix}$

         We can express $\Sigma_j$ as the solution of the equation $X AI A^{T}X^{T}=0$ where $X=(x_1,x_2,x_3,x_4)$.

         Denote $I_1=AI A^{T}$.  Let $\bar{X}=(x_1,x_2,x_4)$ be the projection of $X$ onto the plane $\{h=0\}$ and $\bar{I_1}$ be the $3\times 3$ matrix obtained by deleting the 3rd row and 3rd column of $I_1$.

         Then $\Sigma_j\cap\{h=0\}$ is the solution of $\bar{X}\bar{I_1}\bar{X}^{T}=0$.

         The fact that $A_{13}=O(\eta)$ and $A_{13}=O(\eta)$ implies that $(I_1)_{13}=(I_1)_{31}, (I_1)_{23}=(I_1)_{32}, (I_1)_{43}=(I_1)_{34}$ are all $O(\eta)$ small  and $(I_1)_{33}$ is $O(\eta^2)$ small. Since $I_1$ is rank 2 with two eigenvalues equal to $1$,  $\bar{I_1}$ also has rank 2 with two eigenvalues both in the form of $1+O(\eta)$.
         This means $\Sigma_j\cap\{h=0\}=\mathbb{R}\times E_j$ where $E_j$ is an ellipse with the long and short axis both in the form $1+O(\eta)$. Then the result follows immediately.
       \end{proof}

\end{appendix}




\end{document}